\documentclass{article}
\usepackage[utf8]{inputenc}
\usepackage{format}

\newcommand \bfG{{\mathbf G}}
\newcommand \bfH{{\mathbf H}}
\newcommand \bfT{{\mathbf T}}

\newcommand \bfP{{\mathbf P}}
\newcommand \bfU{{\mathbf U}}
\newcommand \bfL{{\mathbf L}}
\newcommand \Gal{{\mathrm {Gal}}}
\newcommand \sfk{{\mathsf k}}
\newcommand \bark{{\bar{\mathsf k}}}

\counterwithin{paragraph}{subsection}
\numberwithin{equation}{subsection}

\title{The Wavefront Sets of Unipotent Supercuspidal Representations}

\begin{document}
\maketitle
\begin{abstract}
Let $\mathbf{G}(\sfk)$ be a semisimple $p$-adic group, inner to split. In this article, we compute the algebraic and canonical unramified wavefront sets of the irreducible supercuspidal representations of $\mathbf{G}(\sfk)$ in Lusztig's category of unipotent representations. 
\end{abstract}

\tableofcontents

\section{Introduction}

Let $\mathbf{G}$ be a connected semisimple algebraic group defined over a $p$-adic field $\sfk$ with residue field $\mathbb{F}_q$ and let $\mathbf{G}(\sfk)$ be the group of $\sfk$-rational points. Let $X$ be an irreducible admissible representation of $\bfG^\omega(\sfk)$ with distribution character $\Theta_X$. For each nilpotent orbit $\OO$ in the Lie algebra $\fg(\sfk)$ of $\mathbf{G}(\sfk)$, let $\hat{\mu}_{\OO}$ denote the Fourier transform of the associated orbit integral. In \cite{HarishChandra1999}, Harish-Chandra proved that there are complex numbers $c_{\OO}(X) \in \CC$ such that
\begin{equation}\label{eq:localcharacter}\Theta_{X}(\mathrm{exp}(\xi)) = \sum_{\OO} c_{\OO}(X) \hat{\mu}_{\OO}(\xi)\end{equation}
for $\xi \in \fg(\sfk)$ a regular element in a small neighborhood of $0$. The formula (\ref{eq:localcharacter}) is called the \emph{local character expansion} of $X$. 

One of the most fundamental invariants which can be extracted from the local character expansion is the so-called \emph{wavefront set}. This is the set of nilpotent orbits
$$\WF(X) = \max \{\OO \subset \fg(\sfk) \mid c_{\OO}(X) \neq 0\}$$
It is common in the literature to consider a slightly coarser invariant called the \emph{geometric wavefront set}. This is the set of nilpotent orbits over an \emph{algebraic closure} $\bar{\sfk}$ of $\sfk$ which meet some orbit in $\WF(X)$. This set is denoted by $^{\bar{\sfk}}\WF(X)$. A longstanding conjecture of Moeglin and Waldspurger \cite[Page 429]{MW87} asserts that $^{\bar k}\WF(X)$ is a singleton, for all $X$.
Another closely related invariant is the \emph{canonical unramified wavefront set}, which was recently introduced by the third-named author in \cite[Section 2.2.3]{okada2021wavefront}. This final invariant is a refinement of $^{\bar k}\WF(X)$.

Lusztig in \cite[Section 0.3]{Lusztig-IMRN} defined the notion of a \emph{unipotent representation} of $\mathbf{G}(\sfk)$ (see Definition \ref{def:unipotent} below) and completed a Langlands classification for this class of representations, under the assumption that $\bfG$ is simple and adjoint. These restrictions on $\bfG$ have since been removed as a result of various works, culminating with \cite{Sol-LLC}. 

Let $G^{\vee}$ denote the complex Langlands dual group. The (enhanced) Langlands parameter of an irreducible unipotent representation $X$ is the $G^\vee$-orbit of a triple $(\tau,n,\rho)$, where $\tau$ is a semisimple element in $G^\vee$, $n$ is a nilpotent element in $\mathfrak g^\vee$ (the Lie algebra of $G^\vee$) such that $\Ad(\tau)n=q n$, and $\rho$ is an irreducible representation of a certain finite group $A^1_\varphi$, see section \ref{subsec:LLC}. When $\bfG$ is adjoint, $A^1_\varphi$ is the component group of the centralizer of $\tau$ and $n$ in $G^\vee$. Let $\OO^\vee_X$ denote the nilpotent $G^\vee$-orbit of $n$. 

It is natural to ask if and how the local character expansion is related to the Langlands parameter $(\tau,n,\rho)$ of $X$. At one extreme, we have the coefficient $c_0(X)$. It has long been known that when $X$ is tempered, $c_0(X)\neq 0$ if and only if $X$ is square integrable, and in this case, $c_0(X)$ equals, up to a sign, the ratio between the formal degrees of $X$ and the Steinberg representation. An interpretation of the formal degree in terms of the Langlands parameters was conjectured first by Reeder \cite{Reeder-formal}. This interpretation was verified in the case of split exceptional groups by Reeder (\cite{Reeder-formal}) and in the remaining cases by Opdam \cite{Opdam16}. 

At the other extreme, we have the wavefront set of $X$.
In this article, we compute the geometric and canonical unramified wavefront sets of all \emph{supercuspidal} unipotent representations of $\mathbf{G}(\sfk)$ when $\bfG$ is an inner twist of a split group. In all such cases, we find that $^{\bar k}\WF(X)$ is a singleton, and is uniquely determined by the nilpotent part $n$ of the Langlands parameter. More precisely, we prove in Theorem \ref{thm:main}
\begin{equation}\label{eq:mainformulas}
    ^{\bar k}\WF(X)=d(\OO^\vee_X),\quad ^K\WF(X)=d_A(\OO^\vee_X,1).
\end{equation}
Here $d$ and $d_A$ are the duality maps defined by Spaltenstein, also Lusztig and Barbasch-Vogan, and Achar, respectively, see section \ref{subsec:nilpotent}. In particular, the geometric wavefront set $^{\bar k}\WF(X)$ determines the nilpotent orbit $\OO^{\vee}_X$ via the duality map $d$. We emphasize that the simplicity of the formulas (\ref{eq:mainformulas}) is due to the fact that $X$ is supercuspidal and therefore equal to its Aubert-Zelevinsky dual (\cite{Au}). In general, one expects that the wavefront set of $X$ is obtained by duality from the nilpotent parameter associated to the \emph{AZ dual} of $X$, see \cite{unipotent1}. This expression for the wavefront set is closely related to Lusztig's formula for the Kawanaka wavefront set of an irreducible unipotent representation of a finite reductive group \cite[Theorem 11.2]{lusztigunip}. In fact, the finite reductive group results from {\it loc.cit.} play an important role in the construction and analysis of test functions in the local character expansion \cite{barmoy,okada2021wavefront}.

We remark that the methods in this paper show that the canonical unramified wavefront set (and so the geometric wavefront set as well) is in fact a singleton for all depth-zero supercuspidal representations, see Proposition \ref{prop:wfsupp}. The irreducibility of the geometric wavefront set for depth-zero supercuspidal representations was established independently in \cite{AizGouSay}, but also as a consequence of the results in  \cite{barmoy,okada2021wavefront}. 

\section{Preliminaries}\label{sec:preliminaries}

Let $\mathsf{k}$ be a nonarchimedean local field of characteristic $0$ with residue field $\mathbb{F}_q$ of sufficiently large characteristic, ring of integers $\mathfrak{o} \subset \mathsf{k}$, and valuation $\mathsf{val}_{\mathsf{k}}$. Fix an algebraic closure $\bar{\mathsf{k}}$ of $\mathsf{k}$ with Galois group $\Gamma$, and let $K \subset \bar{\mathsf{k}}$ be the maximal unramified extension of $\mathsf{k}$ in $\bar{\mathsf{k}}$. 
Let $\mf O$ be the ring of integers of $K$.
Let $\mathrm{Frob}$ be the geometric Frobenius element of $\mathrm{Gal}(K/\mathsf{k})$, the topological generator which induces the inverse of the automorphism $x\to x^q$ of $\mathbb{F}_q$.

Let $\bfG$ be a connected semisimple algebraic group defined over $\mathbb{Z}$, and let $\bfT \subset \mathbf{G}$ be a maximal torus. For any field $F$, we write $\mathbf{G}(F)$, $\mathbf{T}(F)$, etc. for the groups of $F$-rational points. The $\CC$-points are denoted by $G$, $T$, etc. Let $\bfG_{\ad}=\bfG/Z(\bfG)$ denote the adjoint group of $\bfG$.

Write $X^*(\mathbf{T},\bark)$ (resp. $X_*(\mathbf{T},\bark)$) for the lattice of algebraic characters (resp. co-characters) of $\mathbf{T}(\bark)$, and write $\Phi(\mathbf{T},\bark)$ (resp. $\Phi^{\vee}(\mathbf{T},\bark)$) for the set of roots (resp. co-roots). Let
$$\mathcal R=(X^*(\mathbf{T},\bark), \ \Phi(\mathbf{T},\bark),X_*(\mathbf{T},\bark), \ \Phi^\vee(\mathbf{T},\bark), \ \langle \ , \ \rangle)$$
be the root datum corresponding to $\mathbf{G}$, and let $W$ the associated (finite) Weyl group. Let $\mathbf{G}^\vee$ be the Langlands dual group of $\bfG$, i.e. the connected reductive algebraic group corresponding to the dual root datum 
$$\mathcal R^\vee=(X_*(\mathbf{T},\bark), \ \Phi^{\vee}(\mathbf{T},\bark),  X^*(\mathbf{T},\bark), \ \Phi(\mathbf{T},\bark), \ \langle \ , \ \rangle).$$
Set $\Omega=X_*(\mathbf{T},\bark)/\ZZ \Phi^\vee(\mathbf{T},\bark)$. The center $Z(\bfG^\vee)$ can be naturally identified with the irreducible characters $\mathsf{Irr} \Omega$, and dually, $\Omega\cong X^*(Z(\bfG^\vee))$. For $\omega\in\Omega$, let $\zeta_\omega$ denote the corresponding irreducible character of $Z(\bfG^\vee)$.

For details regarding the parametrization of inner twists of a group $\bfG(\mathsf k)$, see \cite[\S2]{Vogan1993}, \cite{Kottwitz1984}, \cite[\S2]{Kaletha2016}, or \cite[\S1.3]{ABPS2017} and \cite[\S1]{FengOpdamSolleveld2021}. We only record here that the set of equivalence classes of inner twists of the split form of $\mathbf G$ are parametrized by the Galois cohomology group 
\[H^1(\Gamma, \mathbf G_{\ad})\cong H^1(F,\mathbf G_{\ad}(K))\cong\Omega_{\ad}\cong \Irr Z(\bfG^\vee_{\mathsf{sc}}),
\]
where $\bfG^\vee_{\mathsf{sc}}$ is the Langlands dual group of $\bfG_{\ad}$, i.e., the simply connected cover of $\bfG^\vee$, and $F$ denotes the action of $\mathrm{Frob}$ on $\bfG(K)$. We identify $\Omega_{\ad}$ with the fundamental group of $\bfG_{\ad}$. The isomorphism above is determined as follows: for a cohomology class $h$ in $H^1(F, \mathbf G_{\ad}(K))$, let $z$ be a representative cocycle. Let $u\in \bfG_{\ad}(K)$ be the image of $F$ under $z$, and let $\omega$ denote the image of $u$ in $\Omega_{\ad}$. Set $F_\omega=\Ad(u)\circ F$. The corresponding rational structure of $\bfG$ is given by $F_\omega$.
Let $\bfG^\omega$ be the connected semisimple group defined over $\sfk$ such that $\bfG(K)^{F_\omega}=\bfG^\omega(\mathsf k)$.
Note that $\bfG^{1} = \bfG$ (where we view $\bfG$ as an algebraic group over $\sfk$ for this equality).

\

If $H$ is a complex reductive group and $x$ is an element of $H$ or $\fh$, we write $H(x)$ for the centralizer of $x$ in $H$, and $A_H(x)$ for the group of connected components of $H(x)$. If $S$ is a subset of $H$ or $\fh$ (or indeed, of $H \cup \fh$), we can similarly define $H(S)$ and $A_H(S)$. We will sometimes write $A(x)$, $A(S)$ when the group $H$ is implicit. 
The subgroups of $H$ of the form $H(x)$ where $x$ is a semisimple element of $H$ are called \emph{pseudo-Levi} subgroups of $H$.

\medskip

Let $\mathcal C(\bfG(\mathsf k))$ be the category of smooth complex $\bfG(\mathsf k)$-representations and let $\Pi(\mathbf{G}(\mathsf k)) \subset \mathcal C(\bfG(\mathsf k))$ be the set of irreducible objects. Let $R(\bfG(\mathsf k))$ denote the Grothendieck group of $\mathcal C(\bfG(\mathsf k))$.

\subsection{The Bruhat-Tits Building}
\label{sec:btbuilding}

In this section we will recall some standard facts about the Bruhat-Tits building (all of which can be found in \cite{tits}).

Fix a $\omega \in \Omega$ and let $\bfG^\omega$ be the inner twist of $\bfG$ corresponding $\omega$ as defined in the previous section.
Let $\mathcal B(\bfG^\omega,\sfk)$ denote the (enlarged) Bruhat-Tits building for $\bfG^\omega(\sfk)$. 
Let $\mathcal B(\bfG,K)$ denote the (enlarged) Bruhat-Tits for $\bfG(K)$.
For an apartment $\mathcal A$ of $\mathcal B(\bfG,K)$ and $\Omega\subseteq \mathcal A$ we write $\mathcal A(\Omega,\mathcal A)$ for the smallest affine subspace of $\mathcal A$ containing $\Omega$.
The inner twist $\bfG^\omega$ of $\bfG$ gives rise to an action of the Galois group $\Gal(K/k)$ on $\mathcal B(\bfG,K)$ and we can (and will) identify $\mathcal B(\bfG^\omega,\sfk)$ with the fixed points of this action.
We use the notation  $c\subseteq \mathcal B(\bfG^\omega,\sfk)$ to indicate that $c$ is a face of $\mathcal B(\bfG^\omega,\sfk)$.
Given a maximal $\sfk$-split torus $\bfT$ of $\bfG^\omega$, write $\mathcal A(\bfT,\sfk)$ for the corresponding apartment in $\mathcal B(\bfG^\omega,\sfk)$.
Write $\Phi(\bfT,\sfk)$ (resp. $\Psi(\bfT,\sfk)$) for the set of roots of $\bfG(\sfk)$ (resp. affine roots) of $\bfT(\sfk)$ on $\bfG^\omega(\sfk)$. 
For $\psi\in \Psi(\bfT,\sfk)$ write $\dot\psi\in \Phi(\bfT,\sfk)$ for the gradient of $\psi$, and
$W=W(\bfT,\sfk)$ for the Weyl group of $\bfG(\sfk)$ with respect to $\bfT(\sfk)$.
For a face $c\subseteq \mathcal B(\bfG^\omega,\sfk)$ there is a group $\bfP_c^\dagger$ defined over $\mf o$ such that $\bfP_c^\dagger(\mf o)$ identifies with the stabiliser of $c$ in $\bfG(k)$. There is an exact sequence
\begin{equation}\label{eq:parahoricses}
    1 \to \bfU_c(\mf o) \to  \bfP_c^\dagger(\mf o) \to  \bfL_c^\dagger(\mathbb F_q) \to 1,
\end{equation}
where $\bfU_c(\mf o)$ is the pro-unipotent radical of $\bfP_c^\dagger(\mf o)$ and $\bfL_c^\dagger$ is the reductive quotient of the special fibre of $\bfP_c^\dagger$.
Let $\bfL_c$ denote the identity component of $\bfL_c^\dagger$, and let $\bfP_c$ be the subgroup of $\bf P_c^\dagger$ defined over $\mf o$ such that $\bfP_c(\mf o)$ is the inverse image of $\bfL_c(\mathbb F_q)$ in $\bfP_c^\dagger(\mf o)$.
We also write $\bfT$ for the well defined split torus scheme over $\mf o$ with generic fibre $\bfT$.
This scheme $\bfT$ defined over $\mf o$ is a subgroup of $\bfP_c$ and the special fibre of $\bfT$, denoted $\bar\bfT$, is a maximal torus of $\bfL_c$.
For $c$ viewed as a face of $\mathcal B(\bfG,K)$, the stabiliser of $c$ in $\bfG(K)$ identifies with $\bfP_c^\dagger(\mf O)$.
It has pro-unipotent radical $\bfU_c(\mf O)$ and $\bfP_c^\dagger(\mf O)/\bfU_c(\mf O) = \bfL_c^\dagger(\overline{\mathbb F}_q)$.
For $c$ a face lying in $\mathcal B(\bfG^\omega,\sfk)\subseteq \mathcal B(\bfG,K)$, $F_\omega$ stabilises $\bfP_c(\mf O)$ and induces a Frobenius on $\bfL_c(\overline{\mathbb F}_q)$.
The group $\bfL_c(\mathbb F_q)$ consists of the fixed points of this Frobenius.
The groups $\bfP_c(\mf o)$ obtained in this manner are called ($\sfk$-)\emph{parahoric subgroups} of $\bfG^\omega$.
When $c$ is a chamber, then we call $\bfP_c(\mf o)$ an \emph{Iwahori subgroup} of $\mathbf{G}$. 

For this paper it will be convenient to fix a maximal $\sfk$-split torus $\bfT$ of $\bfG^\omega$ and a maximal $K$-split torus $\bfT_1$ of $\bfG^\omega$ containing $\bfT$ and defined over $\sfk$.
We have that $\mathcal A(\bfT,\sfk) = \mathcal A(\bfT_1,K)^{\Gal(K/k)}$.
We will also fix a $\Gal(K/\sfk)$-stable chamber $c_0$ of $\mathcal A(\bfT_1,K)$ and a special point $x_0\in c_0$.
Let $\widetilde W=W\ltimes X_*(\mathbf{T}_1,K)$ be the (extended) affine Weyl group. 
The choice of special point $x_0$ of $\mathcal B(\bfG,K)$ fixes an inclusion $\Phi(\bfT_1,K)\to \Psi(\bfT_1,K)$ and an isomorphism between $\widetilde W$ and $N_{\bfG(K)}(\bfT_1(K))/\bfT_1(\mf O^\times)$.
Write 
\begin{equation}
    \widetilde W\to W, \qquad w\mapsto \dot w
\end{equation}
for the natural projection map.
For a face $c\subseteq \mathcal A$ let $W_c$ be the subgroup of $\widetilde{W}$ generated by reflections in the hyperplanes through $c$.
The special fibre of $\bfT_1$ (as a scheme over $\mf O$) which we denote by $\overline{\bfT}_1$, is a split maximal torus of $\bfL_c(\overline{\mathbb F}_q)$.
Write $\Phi_c(\bar\bfT_1,\overline{\mathbb F}_q)$ for the root system of $\bfL_c$ with respect to $\bar\bfT_1$.
Then $\Phi_c(\bar\bfT_1,\overline{\mathbb F}_q)$ naturally identifies with the set of $\psi\in\Psi(\bfT_1,K)$ that vanish on $c$, and the Weyl group of $\bar\bfT_1$ in $\bfL_c$ is isomorphic to $W_c$. 

Recall that a choice of $x_0$ fixes an embedding $\Phi(\bfT_1,K)\to \Psi(\bfT_1,L)$.
If we fix a set of simple roots $\Delta \subset \Phi(\bfT_1,K)$, this embedding determines a set of extended simple roots $\tilde\Delta\subseteq \Psi(\bfT_1,K)$.
When $\Phi(\bfT_1,K)$ is irreducible, $\tilde\Delta$ is just the set $\Delta\cup\{1-\alpha_0\}$ where $\alpha_0$ is the highest root of $\Phi(\bfT_1,K)$ with respect to $\Delta$.
When $\Phi(\bfT_1,K)$ is reducible, say $\Phi(\bfT_1,K) = \cup_i\Phi_i$ where each $\Phi_i$ is irreducible, then $\tilde\Delta = \cup_i\tilde\Delta_i$ where $\Delta_i = \Phi_i \cap \Delta$.
Fix $\Delta$ so that the chamber cut out by $\tilde\Delta$ is $c_0$.
Let
$$\bfP(\tilde\Delta) := \{J\subsetneq \tilde\Delta: J\cap \tilde\Delta_i\subsetneq \tilde\Delta_i,\forall i\}.$$
Each $J\in \bfP(\tilde\Delta)$ cuts out a face of $c_0$ which we denote by $c(J)$.
In particular $c(\Delta) = x_0$.
Note that since $\Omega \simeq \widetilde W/W\ltimes\ZZ\Phi(\bfT_1,K)$ (recall $\bfG$ is semisimple), and $W\ltimes \ZZ\Phi(\bfT_1,K)$ acts simply transitively on the chambers of $\mathcal A(\bfT_1,K)$, the action of $\widetilde W$ on $\mathcal A(\bfT_1,K)$ induces an action of $\Omega$ on the faces of $c_0$ and hence on $\tilde\Delta$ (and $\bfP(\tilde\Delta)$).
For $\omega\in \Omega$ let $\sigma_\omega$ denote the corresponding permutation of $\tilde\Delta$.
Let
$$\bfP^\omega(\tilde\Delta) := \{J\in \bfP(\tilde\Delta) \mid \sigma_\omega(J) = J\}$$
and let $c_0^\omega$ be the chamber of $\mathcal B(\bfG^\omega,\sfk)$ lying in $c_0$.
The set $\bfP^\omega(\tilde\Delta)$ is an indexing set for the faces of $c_0^\omega$.
For $J\in \bfP^\omega(\tilde\Delta)$ write $c^\omega(J)$ for the face of $c_0^\omega$ corresponding to $J$.
The face $c^\omega(J)$ lies in $c(J)$.
Moreover for $J,J'\in \bfP^\omega(\tilde\Delta)$ (resp. $\bfP(\tilde\Delta)$) we have $J\subseteq J'$ if and only if $\overline{c^\omega(J)}\supseteq c^\omega(J')$ (resp. $\overline{c(J)}\supseteq c(J')$).
\subsection{Nilpotent orbits}\label{subsec:nilpotent}

Let $\mathcal N$ be the functor which takes a field $F$ to the set of nilpotent elements in $\mf g(F)$, and let $\mathcal N_o$ be the functor which takes $F$ to the set of adjoint $\bfG(F)$-orbits on $\mathcal N(F)$. When $F$ is $\sfk$ or $K$, we view $\mathcal N_o(F)$ as a partially ordered set with respect to the closure ordering in the topology induced by the topology on $F$.
When $F$ is algebraically closed, we view $\mathcal N_o(F)$ as a partially ordered set with respect to the closure ordering in the Zariski topology.
For brevity we will write $\mathcal N(F'/F)$ (resp. $\mathcal N_o(F'/F)$) for $\mathcal N(F\to F')$ (resp. $\mathcal N_o(F\to F')$) where $F\to F'$ is a morphism of fields.
For $(F,F')=(\sfk,K)$ (resp. $(\sfk,\bark)$, $(K,\bark)$), the map $\mathcal N_o(F'/F)$ is strictly increasing (resp. strictly increasing, non-decreasing).
We will simply write $\mathcal N$ for $\mathcal N(\CC)$ and $\mathcal N_o$ for $\mathcal N_o(\CC)$.
In this case we also define $\mathcal N_{o,c}$ (resp. $\mathcal N_{o,\bar c}$) to be the set of all pairs $(\OO,C)$ such that $\OO\in \mathcal N_o$ and $C$ is a conjugacy class in the fundamental group $A(\OO)$ of $\OO$ (resp. Lusztig's canonical quotient $\bar A(\OO)$ of $A(\OO)$, see \cite[Section 5]{Sommers2001}). There is a natural map 
\begin{equation}
    \mf Q:\mathcal N_{o,c}\to\mathcal N_{o,\bar c}, \qquad (\OO,C)\mapsto (\OO,\bar C)
\end{equation}
where $\bar C$ is the image of $C$ in $\bar A(\OO)$ under the natural homomorphism $A(\OO)\twoheadrightarrow \bar A(\OO)$. There are also projection maps $\pr_1: \cN_{o,c} \to \cN_o$, $\pr_1: \cN_{o,\bar c} \to \cN_o$. We will typically write $\mathcal N^\vee$, $\mathcal N^\vee_o, \cN^{\vee}_{o,c}$, and $\cN^{\vee}_{o,\bar c}$ for the sets $\mathcal N$, $\mathcal N_o, \cN_{o,c}$, and $\cN_{o,\bar c}$ associated to the Langlands dual group $G^\vee$. When we wish to emphasise the group we are working with we include it as a superscript e.g. $\mathcal N_o^{\bfG^\omega}$.
Note that since $\bfG^\omega$ splits over $K$ we have that $\mathcal N_o^{\bfG}(F) = \mathcal N_o^{\bfG^\omega}(F)$ for field extensions $F$ of $K$.

\paragraph{Classical results and constructions}
Recall the following classical results and constructions related to nilpotent orbits.

\begin{lemma}[Corollary 3.5, \cite{Pommerening} and Theorem 1.5, \cite{Pommerening2}]\label{lem:Noalgclosed}
    Let $F$ be algebraically closed with good characteristic for $\bfG$.
    Then there is canonical isomorphism of partially ordered sets $\Lambda^F:\mathcal N_o^\bfG(F)\xrightarrow{\sim}\mathcal N_o$.
\end{lemma}

Write
\begin{equation}\label{eq:dBV}
d: \cN_0 \to \cN_0^{\vee}, \qquad d: \cN_0^{\vee} \to \cN_0.
\end{equation}
%
duality maps defined by Spaltenstein (\cite[Proposition 10.3]{Spaltenstein}), Lusztig (\cite[\S13.3]{Lusztig1984}, and Barbasch-Vogan (\cite[Appendix A]{BarbaschVogan1985}).
%
Write 
\begin{equation}
    d_S: \cN_{o,c} \twoheadrightarrow \cN^{\vee}_o, \qquad d_S: \cN^{\vee}_{o,c} \twoheadrightarrow \cN_o
\end{equation}
for the duality maps defined by Sommers in \cite[Section 6]{Sommers2001} and 
\begin{equation}
    d_A: \cN_{o,\bar c} \to \cN^{\vee}_{o,\bar c}, \qquad d_A: \cN^{\vee}_{o,\bar c} \to \cN_{o,\bar c}
\end{equation}
for the duality maps defined by Achar in (\cite[Section 1]{Acharduality}). 
These duality maps are compatible in the following sense.
For $\OO\in \mathcal N_o$
$$d_S(\OO,1) = d(\OO)$$
and for $(\OO,C)\in \mathcal N_{o,c}$
$$d_A(\mf Q(\OO,C)) = (d_S(\OO,C),\bar C')$$
for some $\bar C'$.

There is a natural pre-order $\leq_A$ on $\mathcal N_{o,c}$ defined by Achar in \cite[Introduction]{Acharduality} by
$$(\OO,C)\le_A(\OO',C') \iff \OO\le \OO' \text{ and } d_S(\OO,C)\ge d_S(\OO',C').$$
Write $\sim_A$ for the equivalence relation on $\cN_{o,c}$ induced by this pre-order, i.e. 
$$(\OO_1,C_1) \sim_A (\OO_2,C_2) \iff (\OO_1,C_1) \leq_A (\OO_2,C_2) \text{ and } (\OO_2,C_2) \leq_A (\OO_1,C_1)$$
Write $[(\OO,C)]$ for the equivalence class of $(\OO,C) \in \cN_{o,c}$. The $\sim_A$-equivalence classes in $\cN_{o,c}$ coincide with the fibres of the projection map $\mf Q:\mathcal N_{o,c}\to\mathcal N_{o,\bar c}$ \cite[Theorem 1]{Acharduality}. So $\le_A$ descends to a partial order on $\mathcal N_{o,\bar c}$, also denoted by $\le_A$. The maps $d,d_S,d_A$ are all order reversing with respect to the relevant pre/partial orders. 

\paragraph{Structure of $\mathcal N_o^\bfG(K)$}
In \cite[Section 2]{okada2021wavefront} the third-named author establishes a number of results about the structure of $\mathcal N_o^\bfG(K) = \mathcal N_o^{\bfG^\omega}(K)$ which we now briefly summarize.

Let $\bfT$ be a maximal $\sfk$-split torus of $\bfG^\omega$, $\bfT_1$ be a maximal $K$-split torus of $\bfG^\omega$ defined over $\sfk$ and containing $\bfT$, and $x_0$ be a special point in $\mathcal A(\bfT_1,K)$.
In \cite[Section 2.1.5]{okada2021wavefront} the third-named author constructs a bijection
$$\theta_{x_0,\bfT_1}:\mathcal N_o^{\bfG^\omega}(K)\xrightarrow{\sim}\mathcal N_{o,c}.$$
\begin{theorem}
    \label{lem:paramNoK}
    \cite[Theorem 2.20, Theorem 2.27, Proposition 2.29]{okada2021wavefront}
    The bijection 
    $$\theta_{x_0,\bfT_1}:\mathcal N_o^{\bfG^\omega}(K)\xrightarrow{\sim}\mathcal N_{o,c}$$
    is natural in $\bfT_1$, equivariant in $x_0$, and makes the following diagram commute:
    \begin{equation}
        \begin{tikzcd}[column sep = large]
            \mathcal N_o^{\bfG^\omega}(K) \arrow[r,"\theta_{x_0,\bfT_1}"] \arrow[d,"\mathcal N_o(\bar k/K)",swap] & \mathcal N_{o,c} \arrow[d,"\pr_1"] \\
            \mathcal N_o(\bar k) \arrow[r,"\Lambda^{\bar k}"] & \mathcal N_o.
        \end{tikzcd}
    \end{equation}
\end{theorem}

The composition 
$$d_{S,\bfT_1}:= d_S\circ \theta_{x_0,\bfT_1}$$
is independent of the choice of $x_0$ and natural in $\bfT_1$ \cite[Proposition 2.32]{okada2021wavefront}.

For $\OO_1,\OO_2\in \mathcal N_o(K)$ define $\OO_1\le_A\OO_2$ by
$$\OO_1\le_A \OO_2 \iff \mathcal N_o(\bar k/K)(\OO_1) \le \mathcal N_o(\bar k/K)(\OO_2),\text{ and } d_{S,\bfT_1}(\OO_1)\ge d_{S,\bfT_2}(\OO_2)$$
and let $\sim_A$ denote the equivalence classes of this pre-order.
This pre-order is independent of the choice of $\bfT_1$ and the map 
$$\theta_{x_0,\bfT_1}:(\mathcal N_o(K),\le_A) \to (\mathcal N_{o,c},\le_A)$$
is an isomorphism of pre-orders.
\begin{theorem}
    \label{thm:unramclasses}
    The composition $\mf Q\circ \theta_{x_0,\bfT_1}:\mathcal N_o(K)\to \mathcal N_{o,\bar c}$ descends to a (natural in $\bfT_1$) bijection 
    $$\bar\theta_{\bfT_1}:\mathcal N_o(K)/\sim_A\to \mathcal N_{o,\bar c}$$
    which does not depend on $x_0$.
\end{theorem}

\paragraph{Lifting nilpotent orbits}
Define 
\begin{equation}
    \mathcal I_o = \{(c,\OO) \mid c\subseteq \mathcal B(\bfG),\OO\in\cN_o^{\bfL_c}(\overline{\mathbb F}_q)\}.
\end{equation}
There is a partial order on $\mathcal{I}_o$, defined by
$$(c_1,\OO_1)\le(c_2,\OO_2) \iff c_1=c_2 \text{ and } \OO_1\le\OO_2$$
In \cite[Section 1.1.2]{okada2021wavefront} the third author defines a strictly increasing surjective map 
\begin{equation}
    \label{eq:lift}
    \mathscr L:(\mathcal I_o,\le)\to(\mathcal N_o(K),\le).
\end{equation}
The composition 
$$[\bullet]\circ \mathscr L:(\mathcal I_o,\le)\to (\mathcal N_o(K)/\sim_A,\le_A)$$
is also strictly increasing \cite[Corollary 4.7,Lemma 5.3]{okada2021wavefront}.

Recall the groups $G=\bfG(\CC),T=\bfT(\CC)$ from section \ref{sec:preliminaries}.
Call a pseudo-Levi subgroup $L$ of $G$ \emph{standard} if it contains $T$ and write $Z_L$ for its center.
Let $\mathcal A = \mathcal A(\bfT_1,K)$.
\begin{lemma}
    \cite[Section 2.14, Corollary 2.19]{okada2021wavefront}
    There is a $W$-equivariant map
    \begin{equation}
        \mf L_{x_0}:\{\text{faces of } \mathcal A\} \rightarrow\{(L,tZ_L^\circ) \mid L\text{ a standard pseudo-Levi}, \  Z_{G}^\circ(tZ_L^\circ)=L\}
    \end{equation}
    where $c_1,c_2$ lie in the same fibre iff 
    $$\mathcal A(c_1,\mathcal A)+X_*(\bfT_1, K)=\mathcal A(c_2,\mathcal A)+X_*(\bfT_1,K).$$
    Moreover, if $\mf L_{x_0}(c) =(L,tZ_L^\circ)$ then $L$ is the complex reductive group with the same root datum as $\bfL_c(\mathbb F_q)$ and thus there is an isomorphism $\Lambda_c^{\barF_q}:\mathcal N_o^{\bfL_c}(\overline{\mathbb F}_q)\xrightarrow{\sim} \mathcal N_o^{L}$.
\end{lemma}

Recall from the end of section \ref{sec:btbuilding} the definitions of $\Delta,\tilde\Delta,c_0,\bfP(\tilde\Delta)$ and $c(J)$.
Note that the definitions of $c_0$ and $c(J)$ depend on a choice of $x_0$ and $\Delta$.
Let $L_J$ denote the pseudo-Levi subgroup of $G$ generated by $T$ and the root groups corresponding to $\dot \alpha$ for $\alpha\in J$.
Then $\pr_1\circ\mf L_{x_0}(c(J)) = L_J$ (indeed by \cite[Lemma 2.21]{okada2021wavefront}, this group should not depend on $x_0$).
Define 
\begin{align}
    \mathcal{I}_{x_0,\tilde\Delta}&=\{(J,\OO) \mid J\in \bfP(\tilde\Delta), \ \OO\in\mathcal{N}_o^{\bfL_{c(J)}}(\overline{\mathbb F}_q)\}, \\
    \mathcal{K}_{\tilde\Delta}&=\{(J,\OO) \mid J\in \bfP(\tilde\Delta), \ \OO\in\mathcal{N}_o^{L_{J}}(\CC)\}.
\end{align}
The map
$$\iota_{x_0}:\mathcal I_{x_0,\tilde\Delta}\to\mathcal K_{\tilde\Delta},\quad (J,\OO)\mapsto(J,\Lambda_{c(J)}^{\barF_q}(\OO))$$
is an isomorphism.
Let
\begin{equation}
    \mathbb L:\mathcal K_{\tilde\Delta} \to \mathcal N_{c,o}
\end{equation}
be the map that sends $(J,\OO)$ to $(Gx,tZ_{G}^\circ(x))$ where $x\in\OO$ and $(L,tZ_{L}^\circ) = \mf L(c(J))$ where $\mf L$ is the map from \cite[Corollary 2.19]{okada2021wavefront}.
By \cite[Theorem 2.1.7]{unipotent1} the diagram 
\begin{equation}
    \begin{tikzcd}[column sep = large]
        \mathcal I_{x_0,\tilde\Delta} \arrow[r,"\iota_{x_0}"',"\sim"] \arrow[d,two heads,"\mathscr L"] & \mathcal K_{\tilde\Delta}  \arrow[d,two heads,"\mathbb L"] \\
        \mathcal N_o(K) \arrow[r,"\theta_{x_0,\bfT_1}"',"\sim"] & \mathcal N_{o,c}
    \end{tikzcd}
\end{equation}
commutes.
Define 
$$\overline{\mathbb L}=\mf Q\circ \mathbb L.$$
This map can be computed using Achar's algorithms in \cite[Section 3.4]{Acharduality}.

\subsection{Wavefront sets}\label{s:wave}
Let $X$ be an admissible smooth representation of $\bfG^\omega(\sfk)$ and let $\Theta_X$ be the character of $X$.
Recall that for each nilpotent orbit $\OO\in \mathcal N_o^{\bfG^\omega}(\sfk)$ there is an associated distribution $\mu_\OO$ on $C_c^\infty(\mf g^\omega(\sfk))$ called the \emph{nilpotent orbital integral} of $\OO$ \cite{rangarao}.
Write $\hat\mu_\OO$ for the Fourier transform of this distribution.
Generalizing a result of Howe (\cite{Howe1974}), Harish-Chandra in \cite{HarishChandra1999} showed that there are complex numbers $c_{\OO}(X) \in \CC$ such that
\begin{equation}\label{eq:localcharacter}\Theta_{X}(\mathrm{exp}(\xi)) = \sum_{\OO} c_{\OO}(X) \hat{\mu}_{\OO}(\xi)\end{equation}
for $\xi \in \fg^\omega(\sfk)$ a regular element in a small neighborhood of $0$. The formula (\ref{eq:localcharacter}) is called the \emph{local character expansion} of $\pi$. 
The \textit{($p$-adic) wavefront set} of $X$ is
$$\WF(X) := \max\{\OO \mid  c_{\OO}(X)\ne 0\} \subseteq \mathcal N_o(\sfk).$$
The \emph{geometric wavefront set} of $X$ is
$$^{\bar k}\WF(X) := \max \{\mathcal N_o(\bark/\sfk)(\OO) \mid c_{\OO}(X)\ne 0\} \subseteq \mathcal N_o(\bark),$$
see \cite[p. 1108]{Wald18} (warning: in \cite{Wald18}, the invariant $^{\bar k}\WF(X)$ is called simply the `wavefront set' of $X$). 

In \cite[Section 2.2.3]{okada2021wavefront} the third author has introduced a third type of wavefront set for depth-$0$ representations, called
the \emph{canonical unramified wavefront set}. This invariant is a natural refinement of $\hphantom{ }^{\bar{\sfk}}\WF(X)$. We will now define $^K\WF(X)$ and explain how to compute it. 

Recall from Equation \ref{eq:lift} the lifting map $\mathscr L$.
For every face $c \subseteq \mathcal{B}(\mathbf{G})$, the space of invariants $X^{\bfU_c(\mf o)}$ is a (finite-dimensional) $\bfL_c(\mathbb F_q)$-representation. Let $\WF(X^{\bfU_c(\mf o)}) \subseteq \cN^{\bfL_c}_o(\overline{\mathbb F}_q)$ denote the Kawanaka wavefront set \cite{kawanaka} and let 
\begin{equation}
    ^K\WF_c(X) := \{[\mathscr L (c,\OO)] \mid \OO\in \WF(X^{\bfU_{c}(\mf o)})\} \quad \subseteq \cN_o(K).
\end{equation}


\begin{definition}\label{def:CUWF}
    Let $X$ be a depth-0 representation of $\bfG^\omega(\sfk)$.
    The \emph{canonical unramified wavefront set} of $X$ is
    \begin{equation}
        ^K\WF(X) := \max \{\hphantom{ }^K\WF_c(X) \mid  c\subseteq \mathcal B(\bfG^\omega,\sfk)\} \quad \subseteq \cN_o(K)/\sim_A.
    \end{equation}
\end{definition}

Fix $\bfT,\bfT_1,c_0,x_0,\Delta$ as at the end of section \ref{sec:btbuilding}.
By \cite[Lemma 2.36]{okada2021wavefront} we have that
\begin{equation}
    ^K\WF(X) = \max \{\hphantom{ }^K\WF_{c(J)}(X) \mid J \in \bfP^\omega(\tilde\Delta)\}.
\end{equation}
We will often want to view $^K\WF(X)$ and $^K\WF_c(X)$ as subsets of $\mathcal N_{o,\bar c}$ using the identification $\bar\theta_{\bfT_1}$ from Theorem \ref{thm:unramclasses}.
We will write 
$$^{K}\WF(X,\CC) := \bar \theta_{\bfT_1}(\hphantom{ }^K\WF(X)), \quad \hphantom{ }^{K}\WF_c(X,\CC) := \bar \theta_{\bfT_1}(\hphantom{ }^K\WF_c(X)).$$
We will also want to view $\hphantom{ }^{\bar{\sfk}}\WF(X)$, which is naturally contained in $\mathcal N_o(\bar k)$, as a subset of $\mathcal{N}_o$ and so will write 
$$\hphantom{ }^{\bar{\sfk}}\WF(X,\CC) := \Lambda^{\bar \sfk}(\hphantom{ }^{\bar{\sfk}}\WF(X)).$$
By \cite[Theorem 2.37]{okada2021wavefront}, if $^{K}\WF(X)$ is a singleton, then $\hphantom{ }^{\bark}\WF(X)$ is also a singleton and 
\begin{equation}
    \label{eq:wfcompatibility}
    \hphantom{ }^{K}\WF(X,\CC) = (\hphantom{ }^{\bark}\WF(X,\CC),\bar C).
\end{equation}
for some conjugacy class $\bar C$ in $\bar A(\hphantom{ }^\bark\WF(X,\CC))$.

\subsection{Isogenies}
Let $f:\bfH'\to \bfH$ be an isogeny of connected reductive groups defined over $\sfk$.
Let $f_k:\bfH'(k)\to \bfH(k)$ denote the corresponding homomorphism of $k$-points.
We note that $\mathcal N^{\bfH}_{o,\bar c} \simeq \mathcal N^{\bfH'}_{o,\bar c}$ and so we can compare the canonical unramified wavefront sets of representations of the two groups.
In an upcoming paper \cite{isog} the third author proves the following result about the behaviour of the canonical unramified wavefront set under isogeny.
\begin{lemma}
    \label{lem:isogeny}
    Let $X$ be an irreducible admissible depth-0 representation of $\bfH(k)$ and write $X'$ for the representation of $\bfH'(k)$ obtained by pulling back along $f_k$.
    Then $X'$ decomposes as a finite sum of irreducible admissible representations $X' = \bigoplus_i X_i'$ and $^K\WF(X) = \hphantom{ }^K\WF(X_i')$ for all $i$.
\end{lemma}

\subsection{Depth-0 representations}
Let $\omega\in \Omega$.
Recall that a smooth irreducible representation $X$ of $\bfG^\omega(\sfk)$ has depth-0 if $X^{\bfU_c(\mf o)}\ne 0$ for some $c\subseteq \mathcal B(\bfG^\omega,\sfk)$.
Write $\Pi^{0}(\bfG^\omega(\sfk))$ for the subset of $\Pi(\bfG^\omega(\sfk))$ consisting of depth-0 representations.
Let 
$$S^\omega := \{(c,\sigma):c\subseteq \mathcal B(\bfG^\omega,\sfk),\sigma \text{ a cuspidal representation of } \bfL_c(\mathbb F_q)\}$$
and for $(c_1,\sigma_1),(c_2,\sigma_2)\in S^\omega$ write $(c_1,\sigma_1)\sim (c_2,\sigma_2)$ if they are associate in the sense of \cite[Section 5]{moyprasadJ}.
Suppose $(c_1,\sigma_1),(c_2,\sigma_2)\in S^\omega$ are such that $\sigma_i$ is a subrepresentation of $X^{\bfU_{c_i}(\mf o)}$ for $i\in\{1,2\}$.
Then by \cite[Theorem 5.2]{moyprasadJ}, we have that $(c_1,\sigma_1)\sim(c_2,\sigma_2)$.
Thus there is a well defined map 
$$\mathrm{supp}:\Pi^{0}(\bfG^\omega(\sfk))\to S^\omega/\sim$$
which attaches to $X$ the well-defined association class of $(c,\sigma)\in S^\omega$ where $(c,\sigma)$ is such that $\sigma$ appears as a subrepresentation of $X^{\bfU_c(\mf o)}$.
This association class called the unrefined minimal K-type of $X$, but we write $\supp(X)$ for brevity.

Note that for $(c_1,\sigma_1),(c_2,\sigma_2)\in S^\omega$, if $(c_1,\sigma_1)\sim(c_2,\sigma_2)$ then $(c_1,\WF(\sigma_1))\sim_K (c_2,\WF(\sigma_2))$ where $\WF$ denotes the Kawanaka wavefront set and $\sim_K$ is the equivalence relation defined in \cite[Section 1.1.1]{okada2021wavefront}.
In particular, since $\mathscr L$ is constant on $\sim_K$-classes we have that
\begin{equation}
    \label{eq:supporbit}
    \mathscr L(c_1,\WF(\sigma_1)) = \mathscr L(c_2,\WF(\sigma_2)).
\end{equation}
For $X\in \Pi^{0}(\bfG^\omega(\sfk))$ define $^K\WF_{\mathrm{supp}}(X)$ to be $\sim_A$-class of the well-defined orbit in Equation \ref{eq:supporbit} applied to the association class of $\mathrm{supp}(X)$.
We will write 
$$\WFsupp(X,\CC) := \bar\theta_{\bfT_1}(\WFsupp(X)).$$
\begin{lemma}
    For $X\in \Pi^{0}(\bfG^\omega(\sfk))$ we have that
    $$\WFsupp(X)\le_A \OO$$
    for some $\OO\in \hphantom{ }^K\WF(X)$.
    In particular, when $^K\WF(X)$ is a singleton we have
    $$\WFsupp(X)\le_A \hphantom{ }^K\WF(X).$$
\end{lemma}
\begin{proof}
    Let $\supp(X) = [(c,\sigma)]$.
    Then $\sigma$ is a subrepresentation of $X^{\bfU_c(\mf o)}$ and so $\WF(\sigma)\le \OO$ for some $\OO\in \WF(X^{\bfU_c(\mf o)})$.
    Thus
    $$[\mathscr L(c,\WF(\sigma))] \le_A [\mathscr L(c,\OO)] \in \hphantom{ }^K\WF_c(X).$$
    The result then follows from the fact that
    $$^K\WF(X) = \max\{\hphantom{ }^K\WF_c(X):c\subseteq \mathcal B(\bfG^\omega,\sfk)\}.$$
\end{proof}
\begin{rmk}
    In fact the analogous result for the \emph{unramified wavefront set} (as defined in \cite[Section 1.0.5]{okada2021wavefront}), holds as well.
\end{rmk}
For an Iwahori-spherical representation $X$ this inequality says nothing because $\supp(X) = [(c_0,triv)]$ and so $\WFsupp(X,\CC) = (\{0\},1)$.
For supercuspidal representations however, the inequality is in fact an equality.
We now proceed to prove this.

\begin{lemma}
    \label{lem:wf}
    Let $X$ be a depth-0 supercuspidal representation.
    Let $c'$ be a face of $\mathcal B(\bfG^\omega,\sfk)$ with $X^{\bfU_{c'}(\mf o)}\ne 0$ and suppose that $\tau$ is an irreducible constituent of $X^{\bfU_{c'}(\mf o)}$.
    Then $\tau$ is a cupsidal representation of $\bfL_{c'}(\mathbb F_q)$ and in particular $[(c',\tau)] = \mathrm{supp}(X)$.
\end{lemma}
\begin{proof}
    Let $\mathrm{supp}(X) = [(c,\sigma)]$ and $c',\tau$ be as in the statement of the lemma.
    Let $[(M,\tau')]$ be the cuspidal data for $\tau$ (i.e. a conjugacy class of Levi of $\bfL_{c'}(\mathbb F_q)$ and cuspidal representation of said Levi).
    In particular, if $M$ is included into any parabolic $P$ so that $P$ has Levi decomposition $P=MU$, then $\tau'$ is a subrepresentation of $\tau^U$.
    Now, all of the parabolics of $\bfL_{c'}(\mf o)$ are conjugate to a parabolic of the form $\bfP_{c''}(\mf o)/\bfU_{c'}(\mf o)$ where $c'\subseteq \overline{c''}$.
    Thus (conjugating $M$ appropriately) we can find a $c''$ such that $M$ is a Levi factor of $\bfP_{c''}(\mf o)/\bfU_{c'}(\mf o)$ and so $\bfL_{c''}(\mathbb F_q)\simeq M$.
    We thus have that 
    $$\tau' \subseteq \tau^U \subseteq (X^{\bfU_{c'}(\mf o)})^{\bfU_{c''}(\mf o)/\bfU_{c'}(\mf o)} = X^{\bfU_{c''}(\mf o)}.$$
    In particular $(c'',\tau')$ is an unrefined minimal $K$-type for $X$.
    Thus by \cite[Theorem 5.2]{moyprasadJ}, we have that $(c'',\tau')\sim(c,\sigma)$.
    In particular $c''$ is also a minimal face and so $c'' = c'$ and $\tau = \tau'$.
    Thus $\tau$ is a cuspdial representation of $\bfL_{c'}(\mathbb F_q)$ and $[(c,\sigma)] = [(c',\tau)]$.
\end{proof}

\begin{prop}
    \label{prop:wfsupp}
    Let $X$ be a depth-0 supercuspidal representation of $\bfG^\omega(\sfk)$.
    Then 
    $$\WFsupp(X) = \hphantom{ }^K\WF(X).$$
\end{prop}
\begin{proof}
    Recall that
    $$^K\WF(X) = \max\{\hphantom{ }^K\WF_c(X):c\subseteq \mathcal B(\bfG^\omega,\sfk)\}.$$
    Suppose $c\subseteq \mathcal B(\bfG^\omega,\sfk)$ is such that $X^{\bfU_c(\mf o)}\ne 0$.
    Then for any $\tau$ an irreducible constituent of $X^{\bfU_c(\mf o)}$ we have by Lemma \ref{lem:wf} that $[(c,\tau)] = \supp(X)$.
    Thus we must have that $^K\WF_c(X) = \WFsupp(X)$.
    If $c\subseteq \mathcal B(\bfG^\omega,\sfk)$ is such that $X^{\bfU_c(\mf o)} = 0$ then $^K\WF_c(X,\CC) = (\{0\},1)$.
    Thus 
    $$^K\WF(X) = \WFsupp(X).$$
\end{proof}
\begin{rmk}
    Although we have only stated the results for inner twists of the split group, the proposition in fact holds for all connected reductive groups (with identical proof).
\end{rmk}
Note that since $\WFsupp$ is a singleton by construction, this establishes Moeglin and Waldspurger's conjecture for all depth-0 supercuspidal representations for connected reductive groups.
\subsection{Unipotent supercuspidal representations}
\label{s:unip-cusp}
Let 
$$S_{unip}^\omega := \{(c,\sigma)\in S^\omega : \sigma \text{ is \emph{unipotent}}\}.$$

\begin{definition}\label{def:unipotent} Let $X$ be an irreducible $\mathbf{G}^\omega(\mathsf{k})$-representation. We say that $X$ has \emph{unipotent cuspidal support} if $\supp(X)\in S_{unip}^\omega$. Write $\Pi^{\mathsf{Lus}}(\bfG^\omega(\mathsf k))$ for the subset of $\Pi(\bfG(\mathsf k))$ consisting of all such representations.
\end{definition}

Call a supercuspidal representation unipotent if it has cuspidal unipotent support.
By \cite{Morris1996}, the irreducible unipotent supercuspidal representations are all obtained by compact induction $\mathrm{ind}_{\mathbf{P}_c^\dagger(\mf o)}^{\bfG(\mathsf k)}(\sigma^\dagger)$, where $\mathbf{P}_c(\mf o)$ is a maximal parahoric subgroup and $\sigma^\dagger$ contains an irreducible Deligne-Lusztig cuspidal unipotent representation $\sigma$ of ${\mathbf L}_c(\mathbb F_q)$ upon restriction to $\bfL_c(\mathbb F_q)$. 
In particular $X\in \Pi^{\mathsf{Lus}}(\bfG^\omega(\sfk))$ is supercuspidal if and only if $\mathrm{supp}(X) = [(c,\sigma)]$ where $c$ is a minimal face of $\mathcal B(\bfG^\omega,\sfk)$.

\subsection{Langlands classification of unipotent supercuspidal representations}\label{subsec:LLC}
Let $W_{\mathsf{k}}$ be the Weil group of $\mathsf k$ with inertia subgroup $I_{\mathsf k}$ and set $W_{\mathsf k}'=W_{\mathsf k}\times SL(2,\CC)$. We will think of a \emph{Langlands parameter} for $\bfG$  as  a continuous morphisms $\varphi: W_{\mathsf k}'\to G^\vee$ such that $\varphi(w)$ is semisimple for each $w\in W_{\mathsf k}$ and the restriction of $\varphi$ to $SL(2,\CC)$ is algebraic. A Langlands parameter $\varphi$ is called \emph{unramified} if $\varphi(I_{\mathsf k})=\{1\}$. Let $G^\vee(\varphi)$ denote the centralizer of $\varphi(W_{\mathsf k}')$ in $G^\vee$. Define
\[Z^1_{G^\vee_{\mathsf{sc}}}(\varphi)=\text{ preimage of }G^\vee(\varphi)/Z(G^\vee)\text{ under the projection } G^\vee_{\mathsf{sc}}\to G^\vee_{\ad},
\]
and let $A^1_\varphi$ denote the component group of $Z^1_{G^\vee_{\mathsf{sc}}}(\varphi)$. An \emph{enhanced Langlands parameter} is pair $(\varphi,\rho)$, where $\rho \in\mathrm{Irr}(A^1_\varphi).$ A parameter $(\varphi,\rho)$ is called \emph{$\bfG^{\omega}$-relevant} (recall that $\bfG^{\omega}$ is an inner twist of the split form, $\omega\in \Omega_{\ad}$) if $\rho$ acts on $Z(G^\vee_{\mathsf{sc}})$ by a multiple of the character $\zeta_\omega$.

Define the elements
\[s_\varphi=\varphi(\mathrm{Frob},1),\quad u_\varphi=\varphi(1,\begin{pmatrix} 1&1\\0&1\end{pmatrix}).
\]
Following \cite{AMSol2018}, consider the possibly disconnected reductive group 
\[\mathcal G_\varphi=Z^1_{G^\vee_{\mathsf{sc}}}(\varphi(W_{\mathsf{k}})),
\]
which is defined analogously to $Z^1_{G^\vee_{\mathsf{sc}}}(\varphi)$. Then $u_\varphi\in \mathcal G_\varphi^\circ$ and by \cite[(92)]{AMSol2018}
\[A^1_\varphi\cong \mathcal G_\varphi(u_\varphi)/\mathcal G_\varphi(u_\varphi)^\circ.
\]
An enhanced Langlands parameter $(\varphi,\rho)$ is called \emph{discrete} if $G^{\vee}(\varphi)$ does not contain a nontrivial torus (this notion is independent of $\rho$). A discrete parameter is called \emph{cuspidal} if $(u_\varphi,\rho)$ is a cuspidal pair. This means that every $\rho^\circ$ which occurs in the restriction of $\rho$ to $A_{\mathcal G_\varphi^\circ}(u_\varphi)$ defines a $\mathcal G_\varphi^\circ$-equivariant local system on the $\mathcal G_\varphi^\circ$-conjugacy class of $u_\varphi$ which is cuspidal in the sense of Lusztig. 

A Langlands correspondence for unipotent supercuspidal representations has been obtained by \cite{Morris1996} when $\bfG$ is simple and adjoint, see also \cite{Lusztig-IMRN}. For arbitrary reductive $K$-split groups, this correspondence is available by \cite{FengOpdamSolleveld2021,FengOpdam2020}. Let $\mathrm{Irr}(\bfG^{\omega}(\sfk))_{\mathrm{cusp,unip}}$ denote the set of equivalence classes of irreducible unipotent supercuspidal $\bfG^{\omega}(\sfk)$-representations. Let $\Phi(G^\vee)_{\mathrm{cusp,nr}}^\omega$ denote the set of $G^\vee$-equivalence classes of unramified cuspidal enhanced Langlands parameters $(\varphi,\rho)$ which are $\bfG^{\omega}$-relevant.

\begin{theorem}For every $\omega\in\Omega_{\ad}$, there is a bijection
\[\Phi(G^\vee)_{\mathrm{cusp,nr}}^\omega\longleftrightarrow\mathrm{Irr}(\bfG^{\omega}(\sfk))_{\mathrm{cusp,unip}}.
\]
This bijection satisfies several natural desiderata (including formal degrees, equivariance with respect to tensoring by weakly unramified characters), see \cite[Theorem 2]{FengOpdamSolleveld2021}.
\end{theorem}

For $X$ a unipotent supercuspidal representation of $\bfG^\omega(\sfk)$ let $\varphi$ denote the corresponding Langlands parameter.
We will write $\OO^\vee_X\in\mathcal N_o^\vee$ for the $G^\vee$-orbit of
$$n_\varphi = d\varphi(0,\begin{pmatrix} 0&1\\0&0\end{pmatrix}).$$
\begin{lemma}
    Let $X,X'\in \mathrm{Irr}(\bfG^\omega(\sfk))_{\mathrm{cusp,unip}}$.
    If $\supp(X) = \supp(X')$ then $\OO^\vee_X = \OO^\vee_{X'}$.
\end{lemma}
\begin{proof}
    This follows by inspecting the explicit classification in \cite{FengOpdamSolleveld2021}.
\end{proof} 

For $[(c,\sigma)]\in S^\omega$ with $c$ a minimal face we write $\OO^\vee(c,\sigma)$ for the common nilpotent parameter of all $X\in \mathrm{Irr}(\bfG^\omega(\sfk))_{\mathrm{cusp,unip}}$ with $\supp(X) = [(c,\sigma)]$.

We will recall the explicit classification in the section \ref{sec:Langlandscuspidal}.

\section{Main result}

\begin{prop}\label{prop:main}
Suppose $\mathbf{G}$ is simple and adjoint and let $[(c,\sigma)]\in S^\omega$ be such that $c$ is a minimal face.
Then 
$$\overline{\mathbb{L}}(c,\WF(\sigma)) = d_A(\OO^{\vee}(c,\sigma),1).$$
\end{prop}

This proposition will be proved in Section \ref{sec:proof}.

\begin{theorem}
    \label{thm:main}
    Let $\bfG$ be a split semisimple group defined over $\sfk$.
    Let $\omega\in \Omega$ and let $\bfG^\omega$ denote the corresponding inner twist of $\bfG$.
    Let $X$ be an irreducible supercuspidal $\mathbf{G}^\omega(\sfk)$-representation with unipotent cuspidal support. Then
     $^K\WF(X)$, $^{\bark}\WF(X)$ are singletons, and
    \begin{align*}
        ^K\WF(X,\CC) &= d_A(\OO^{\vee}_X,1)\\
        ^{\bark}\WF(X,\CC) &= d(\OO^{\vee}_X).
    \end{align*}
\end{theorem}
\begin{proof}
    Suppose first that $\mathbf{G}$ is simple and adjoint.     
    
    Let $X$ be a unipotent supercuspidal representation of $\bfG^\omega(\sfk)$.
    By Proposition \ref{prop:wfsupp} we have that
    $$^K\WF(X) = \WFsupp(X).$$
    Write $\supp(X) = [(c,\sigma)]$.
    By definition
    $$\WFsupp(X,\CC) = \overline{\mathbb L}(c,\WF(\sigma)).$$
    By Proposition \ref{prop:main} we have that
    $$\overline{\mathbb{L}}(c,\WF(\sigma)) = d_A(\OO^{\vee}(c,\sigma),1).$$
    Since $\OO^\vee_X = \OO^\vee(c,\sigma)$ we get that
    $$^K\WF(X;\CC) = d_A(\OO^{\vee}_X,1)$$
    as required.
    
    Applying Lemma \ref{lem:isogeny} we get that the theorem holds for all simply-connected simple groups.
    Since wavefront sets behave as expected with respect to products, the theorem holds for all simply-connected semisimple groups.
    Finally, applying Lemma \ref{lem:isogeny} again we get that the theorem holds for all split semisimple groups $\bfG$.
\end{proof}

\section{Unipotent cuspidal representations of finite reductive groups}
\label{sec:unipcusp}

For the explicit results about the parametrization of unipotent representations of finite reductive groups, we refer to \cite[\S4,\S8.1]{Lusztig1984} and \cite[\S13.8,\S13.9]{Carter1993}. The relevant results for the Kawanaka wavefront sets and unipotent support are in \cite[\S10,\S11]{lusztigunip}. The classification of unipotent representations is independent of the isogeny, so in this section, we may assume without loss of generality that the group $G$ is simple and adjoint. 

\subsection{Classical groups}

\paragraph{$A_{n-1}(q)$.} The group $G=PGL(n)$ does not have unipotent cuspidal representations. 

\paragraph{$^2\!A_n(q^2)$.} The group $G=PU(n+1)$ has unipotent representations if and if $n=\frac{r(r+1)}2-1$, for some integer $r\ge 2$. The unipotent $^2\!A_n(q^2)$-representations are in one-to-one correspondence with partitions of $n+1$, and so are the geometric nilpotent orbits of $G$. When $n=\frac{r(r+1)}2-1$, the cuspidal unipotent representation $\sigma$ is unique and it is parametrized by the partition
\[(1,2,3,\dots,r).
\]
The Kawanaka wavefront set if $\WF(\sigma)=(1,2,3,\dots,r).$

\paragraph{$B_n(q),C_n(q)$.} Suppose $G$ is $SO(2n+1)$ or $PSp(2n)$ over $\mathbb F_q$. The group $G(\FF_q)$ has a unipotent cuspidal representation (and in this case the cuspidal representation is unique) if and only if $n=r^2+r$ for a positive integer $r$. The unipotent representations of $G(\FF_q)$ are parametrized by symbols
\[\begin{pmatrix} \lambda_1 & &\lambda_2&\dots & &\lambda_a\\&\mu_1&&\dots&\mu_b
\end{pmatrix},
\]
$0\le\lambda_1<\lambda_2<\dots<\lambda_a$, $0\le \mu_1<\mu_2<\dots<\mu_b$, $a-b$ odd and positive, and $\lambda_1,\mu_1$ are not both zero, such that
\[n=\sum\lambda_i+\sum\mu_j-\left(\frac{a+b-1}2\right)^2.
\]
Let $d=a-b$ be the defect of the symbol. Two unipotent representations belong to the same family if their symbols have the same entries with the same multiplicities. 
For the unipotent cuspidal representation $\sigma$, the corresponding symbol has defect $d=2r+1$ and it  is
\[\begin{pmatrix}
0&1&2&\dots&2r\\
&&-
\end{pmatrix}.
\]
The geometric nilpotent orbits of $SO(2n+1)$ (resp., $PSp(2n)$) are parametrized by partitions of $2n+1$ (resp., $2n$), where the even (resp., odd) parts occur with even multiplicity. The Kawanaka wavefront set of the unipotent cuspidal representation $\sigma$ is
\begin{equation}
\WF(\sigma)=\begin{cases}(1,1,3,3,\dots,2r-1,2r-1,2r+1),&G=SO(2n+1)\\
(2,2,4,4,\dots,2r,2r), &G=PSp(2n).
\end{cases}
\end{equation}

\paragraph{$D_n(q)$.} Suppose $G$ is the split orthogonal group $PSO(2n)$ over $\FF_q$.  
There exists a unipotent cuspidal representation (and in this case it is unique) if and only if $n=r^2$ for a positive even integer $r$. 
The type $D_n$-symbols are
\[\begin{pmatrix} \lambda_1 & &\lambda_2&\dots & &\lambda_a\\&\mu_1&&\dots&\mu_b
\end{pmatrix},
\]
$0\le\lambda_1<\lambda_2<\dots<\lambda_a$, $0\le \mu_1<\mu_2<\dots<\mu_b$, $a-b$ is divisible by $4$, and $\lambda_1,\mu_1$ are not both zero, such that
\[n=\sum\lambda_i+\sum\mu_j-\frac {(a+b)(a+b-2)}4.
\]
One symbol and the symbol if the row swapped are regarded the same. 
The irreducible unipotent $G(\FF_q)$-representations are in one-to-one correspondence with the type $D_n$-symbols, except if the symbol has identical rows, then there are two nonisomorphic irreducible unipotent representations attached to it. 
The defect $d=a-b$ is even.

For the unipotent cuspidal representation $\sigma$, the corresponding symbol has defect $d=2r$ and it is
\[\begin{pmatrix}
0&1&2&\dots&2r-1\\
&&-
\end{pmatrix}.
\]
The geometric nilpotent orbits of $PSO(2n)$ are parametrized by partitions of $2n$ with the even parts occuring with even multiplicity. The Kawanaka wavefront set of the unipotent cuspidal representation is 
\begin{equation}\label{e:D}
\begin{aligned}
\WF(\sigma)=(1,1,3,3,\dots,2r-1,2r-1).
\end{aligned}
\end{equation}

\paragraph{$^2\!D_n(q^2)$.} 
The group $^2\!D_n(q^2)$ admits unipotent cuspidal representations if and only if $n=r^2$, for some odd positive integer $r$, and in this case the unipotent cuspidal representation is unique. 
The type $^2\!D_n$-symbols are
\[\begin{pmatrix} \lambda_1 & &\lambda_2&\dots & &\lambda_a\\&\mu_1&&\dots&\mu_b
\end{pmatrix},
\]
$0\le\lambda_1<\lambda_2<\dots<\lambda_a$, $0\le \mu_1<\mu_2<\dots<\mu_b$, $a-b\equiv 2$ mod $4$, and $\lambda_1,\mu_1$ are not both zero, such that
\[n=\sum\lambda_i+\sum\mu_j-\frac {(a+b)(a+b-2)}4.
\]
One symbol and the symbol if the row swapped are regarded the same. The irreducible unipotent $^2\!D_n(q^2)$-representations are in one-to-one correspondence with the type $^2\!D_n$-symbols.

For the unipotent cuspidal representation $\sigma$, the corresponding symbol and Kawanaka wavefront set are the same as in the split case $D_n(q)$ (except $r$ is now odd).

\paragraph{$^3\!D_4(q^3)$} 
The group $^3\!D_4(q^3)$ has eight unipotent representations: six are in the principal series, in one-to-one correspondence with the irreducible representations of the Weyl group of type $G_2$, and two unipotent cuspidal representations, denoted $^3\!D_4[1]$ and $^3\!D_4[-1]$. 

The geometric nilpotent orbits of $^3\!D_4(q^3)$ are parametrized by partitions of $8$ with even parts occurring with even multiplicity. The unipotent cuspidal representations have Kawanaka wavefront set
\begin{equation}
    \WF(^3\!D_4[1])=\WF(^3\!D_4[1])=(1,1,3,3).
\end{equation}

\subsection{Exceptional groups}

Suppose $G(\FF_q)$ is $\FF_q$-split. In the table below, we will list all unipotent cuspidal $G(\FF_q)$-representations. The irreducible unipotent representations of $G(\FF_q)$ are partitioned into families, each family being in one-to-one correspondence with the set
\[M(\Gamma)=\Gamma\text{-orbits in }\{(x,\tau)\mid x\in \Gamma,\ \tau\in\widehat{Z_\Gamma(x)}\},
\]
for a finite group $\Gamma$. Each group $\Gamma$ is uniquely attached to a special nilpotent orbit $\OO^\vee$ in the dual Lie algebra, such that $\Gamma=\bar{A}(\OO^{\vee})$, where $\bar{A}(\OO^{\vee})$ is Lusztig's canonical quotient. 

In Table \ref{ta:exc-fin}, for each  unipotent cuspidal representation $\sigma$, we will record the corresponding Kawanaka wavefront set, the nilpotent orbit $\OO^{\vee}$ corresponding to $\sigma$ and its canonical quotient $\bar{A}(\OO^{\vee})$, the pair $(x,\tau)\in M(\bar{A}(\OO^{\vee}))$ that parametrizes $\sigma$. The geometric nilpotent orbits are given in the Bala-Carter notation.


\begin{small}
\begin{longtable}{|c|c|c|c|c|c|}

    \hline
        $G(\FF_q)$ & Cuspidal $\sigma$ &$\WF(\sigma)$ & $\OO^\vee_{\sigma}$   &$\bar A(\OO^\vee)$ &$(x,\tau)$\\ \hline
\hline
$G_2$ &$G_2[1]$ &$G_2(a_1)$ &$G_2(a_1)$ &$S_3$ &$(1,\epsilon)$\\
           &$G_2[-1]$ & & & &$(g_2,\epsilon)$\\
            &$G_2[\theta^l]$, $l=1,2$ & & & &$(g_3,\theta^l)$\\ \hline
 $F_4$ &$F_4^{II}[1]$ &$F_4(a_3)$ &$F_4(a_3)$ &$S_4$ &$(1,\lambda^3)$\\
           &$F_4[-1]$   & & & &$(g_2,\epsilon)$\\
           &$F_4^I[1]$ & & & &$(g_2',\epsilon)$\\
           &$F_4[\theta^l]$, $l=1,2$ & & & &$(g_3,\theta^l)$\\
           &$F_4[\pm i]$           & & & &$(g_4,\pm i)$\\ 
\hline
$E_6$ &$E_6[\theta^l]$, $l=1,2$ &$D_4(a_1)$ &$D_4(a_1)$ &$S_3$ &$(g_3,\theta^l)$\\
\hline
$E_7$ &$E_7[ \zeta]$ &$A_4+A_1$ &$A_4+A_1$ &$\ZZ/2$ &$(g_2,1)$\\
&$E_7[ \zeta]$ & & & &$(g_2,\epsilon)$\\
\hline
$E_8$ &$E_8^{II}[1]$ &$E_8(a_7)$ &$E_8(a_7)$ &$S_5$ &$(1,\lambda^4)$\\
           &$E_8[-1]$ & & & &$(g_2,-\epsilon)$\\
           &$E_8^{I}[1]$ & & &  &$(g_2',\epsilon)$\\
           &$E_8[\theta^l]$, $l=1,2$ & & & &$(g_3,\epsilon\theta^l)$\\
           &$E_8[-\theta^l]$, $l=1,2$ & & & &$(g_6,-\theta^l)$\\
           &$E_8[\pm i]$ & & & &$(g_4,\pm i)$\\
           &$E_8[\zeta^j]$, $1\le j\le 4$ & & & &$(g_5,\zeta^j)$\\
           \hline

\caption{Unipotent cuspidal representations of exceptional groups $G(\FF_q)$}\label{ta:exc-fin}
\label{table:cupsidalexceptional}

\end{longtable}
\end{small}

Finally, for the twisted group $^2\!E_6(q^2)$, there are three unipotent cuspidal representations, denoted $^2\!E_6[1]$, $^2\!E_6[\theta]$, $^2\!E_6[\theta^2]$. All three of them have Kawanaka wavefront set $D_4(a_1)$ in $E_6$.

\section{Langlands parameters for unipotent supercuspidal representations}\label{sec:Langlandscuspidal}

Recall that $X\in \Pi^{\mathsf{Lus}}(\bfG^\omega(\sfk))$ is supercuspidal if and only if $\supp(X) = [(c,\sigma)]$ where $c$ is a minimal face of $\mathcal B(\bfG^\omega,\sfk)$.
Since every association class of faces of $\mathcal B(\bfG^\omega,\sfk)$ contains a face of $c_0^\omega$ we may assume that $c$ is of the form $c^\omega(J)$ for some $J\in \bfP^\omega(\tilde\Delta)$.
Moreover $c^\omega(J)$ is a minimal face if and only if $J$ is maximal in $\bfP^\omega(\tilde\Delta)$.
In this section we list the set of possible pairs $(J,\sigma)$ (up to $\sim$) along with $\OO^\vee(c^\omega(J),\sigma)$, where $J\in \bfP^\omega(\tilde\Delta)$ is maximal and $\sigma$ is a unipotent cupsidal representation of $\bfL_{c^\omega(J)}(\mathbb F_q)$, for $\bfG$ split, simple and adjoint, and $\omega\in \Omega$.
We use the conventions of \cite[Section 6.10]{Lusztig-IMRN} to specify the set $J\subseteq \tilde\Delta$.
Note that when $\bfG$ is of classical type, the group $\bfL_{c^\omega(J)}(\mathbb F_q)$ is also of classical type and so if it admits a unipotent cuspidal representation, then it has \emph{exactly one} unipotent cuspidal representation.
Thus for the classical types we will only record the $J$ and $\OO^\vee(c^\omega(J),\sigma)$.
The explicit parameters can be found in \cite{Lusztig-IMRN}, \cite{Reeder-formal}, \cite[\S4.7]{FengOpdam2020}, and  \cite{FengOpdamSolleveld2021}.

\subsection{Classical groups}
\paragraph{$PGL(n)$} If $\bfG=PGL(n)$, then $G^\vee=SL(n,\CC)$ and $Z(G^\vee)=\ZZ/n\ZZ$. Hence $\Omega=\text{Irr}(Z(G^\vee)$ can be identified with $C_n$. For $\omega\in \Omega$, the inner form $\bfG^{\omega}$ admits unipotent supercuspidal representations if and only if $\omega$ has order $n$ and $J=\emptyset$. In this case $\OO^\vee(c^\omega(J),\sigma)$ is the principal nilpotent orbit.

\paragraph{$SO(2n+1)$} 
\label{sec:typeB}
If $\bfG=SO(2n+1)$, $G^\vee=Sp(2n,\CC)$ and $Z(G^\vee)=\ZZ/2\ZZ$. The inner forms are parametrized by $\widehat{Z(G^\vee)}\cong C_2=\{1,-1\}$. 

\begin{enumerate}
    \item If $\omega = 1$, then $J$ is of the form $D_\ell\times B_t$, where $\ell+t=n$, $\ell=a^2$, $t=b(b+1)$, $a,b$ nonnegative integers, $a$ even.
Let
\begin{equation}
    \delta = \begin{cases}
        b-a & \mbox{ if } b\ge a \\
        a-b-1 & \mbox{ if } a>b,
    \end{cases}
\end{equation}
and $\Sigma = a+b$.
The nilpotent orbit $\OO^\vee(c^\omega(J),\sigma)$ is parameterized by the partition
\begin{equation}
    \lambda= (2,4,\dots,2\delta) \cup (2,4,\dots,2\Sigma).    
\end{equation}

    \item If $\omega = -1$, then $J$ is of the form $D_\ell\times B_t$, where $\ell+t=n$, $\ell=a^2$, $t=b(b+1)$, $a,b$ nonnegative integers, where $a$ is now odd. The nilpotent orbit $\OO^\vee(c^\omega(J),\sigma)$ is defined analogously to the $\omega=1$ case.
\end{enumerate}

\paragraph{$PSp(2n)$}
\label{sec:typeC}
If $\bfG=PSp(2n)$, then $G^\vee=Spin(2n+1,\CC)$, and $Z(G^\vee)=\ZZ/2\ZZ$. The inner forms are parametrized by $\widehat{Z(G^\vee)}\cong C_2=\{1,-1\}$. 

\begin{enumerate}
    \item If $\omega = 1$, then $J$ is of the form $C_\ell\times C_t$, where $\ell + t = n$, $\ell=a(a+1)$, $t=b(b+1)$, $a,b$ nonnegative integers and $a\ge b$. 
Let $\delta = a-b$ and $\Sigma = a+b$. 
The nilpotent orbit $\OO^\vee(c^\omega(J),\sigma)$ is parameterized by the partition
\begin{equation}
    \lambda = (1,3,\dots,2\delta-1) \cup (1,3,\dots,2\Sigma+1) 
\end{equation}
where $\cup$ means union of partitions.

    \item If $\omega = -1$, then $J$ is of the form $J=C_\ell~^2\!A_{t} ~C_\ell$, where $2\ell+t = n-1$ and $t=\frac {a(a+1)}2-1$, $\ell=b(b+1)$, $a,b$ are nonnegative integers. If $a=0,1$, we interpret $J$ as being $J=C_\ell\times C_\ell$.
Let $a'$ be such that $a=2a'$ if $a$ is even and $a=2a'+1$ if $a$ is odd.
Let $\Sigma = b+a'$ and
$$\delta = \begin{cases}
    b-a' & \mbox{ if } 2b\ge a \\
    a'-b & \mbox{ if } 2b< a.
\end{cases}$$
The nilpotent orbit $\OO^\vee(c^\omega(J),\sigma)$ is parameterized by the partition
$$\lambda = \begin{cases}
    (1,5,\dots,4\Sigma+1) \cup (3,7,\dots,4\delta-1) & \mbox{ if $a$ is even  and } 2b\ge a \\
    (1,5,\dots,4\Sigma+1) \cup (1,5,\dots,4\delta-3) & \mbox{ if $a$ is even  and } 2b< a \\
    (3,7,\dots,4\Sigma+3) \cup (1,5,\dots,4\delta-3) & \mbox{ if $a$ is odd and } 2b\ge a \\
    (3,7,\dots,4\Sigma+3) \cup (3,7,\dots,4\delta-1) & \mbox{ if $a$ is odd and } 2b< a.
\end{cases}$$
\end{enumerate}

\paragraph{$PSO(2n)$}
\label{sec:typeD}
If $\bfG=PSO(2n)$, then $G^\vee=Spin(2n,\CC)$, and 
\[ Z(G^\vee)=\begin{cases}(\ZZ/2\ZZ)^2, &n\text{ even},\\ \ZZ/4\ZZ,&n\text{ odd}.\end{cases}
\]
Let $\tau$ be the standard diagram automorphism of type $D_n$. Let $\{1,-1\}$ be the kernel of the isogeny $Spin(2n,\CC)\to SO(2n,\CC)$. Write the four characters of $Z(G^\vee)$ as $\Omega=\{1,\eta,\rho,\eta\rho\}$, where $\tau(\eta)=\eta$ and $\eta(-1)=1$. 

\begin{enumerate}
    \item If $\omega=1$, then $J$ is of the form $D_\ell\times D_t$, where $\ell+t=n$, $\ell=a^2$, $t=b^2$, $a,b$ even non-negative integers, $a\ge b$.
    Let $\delta = a-b,\Sigma = a+b$.
    The nilpotent orbit $\OO^\vee(c^\omega(J),\sigma)$ is parameterized by the partition
\begin{equation}
    \lambda = (1,3,\dots,2\delta-1)\cup (1,3,\dots, 2\Sigma-1).
\end{equation}

\item If $\omega=\eta$, then $J$ is of the form $\hphantom{ }^2D_\ell\times \hphantom{ }^2D_t$, where $\ell+t=n$, $\ell=a^2$, $t=b^2$, and $a,b$ are odd positive integers, $a\ge b$. The nilpotent orbit $\OO^\vee(c^\omega(J),\sigma)$ is defined analogously to the $\omega=1$ case.

\item If $\omega=\rho, \eta\rho$, then $J$ can take one of the following two forms:
\begin{enumerate}
\item[(i)] $J$ is of the form $~^2\!A_{t}$, where $t = n-1$ is even, $t=\frac{a(a+1)}2-1$, $a$ is a non-negative integer. This means that $a\equiv 0,3$ (mod $4$). There are four ways to embed $J$ into the affine Dynkin diagram $\widetilde D_n$, two of them are $\rho$-stable, and the other two $\eta\rho$-stable. 
In all cases the nilpotent orbit $\OO^\vee(c^\omega(J),\sigma)$ is parameterized by the partition
\begin{equation}
    \lambda = \begin{cases}
        (3,3,7,7,\dots,2a-1,2a-1),&a\equiv 0 \ (\text{mod }4),\\
        (1,1,5,5,\dots,2a-1,2a-1),&a\equiv 3\  (\text{mod }4).
    \end{cases}
\end{equation}
\item[(ii)] $J$ is of the form $D_\ell~ ^2\! A_t~D_\ell$, where $2\ell+t=n-1$, $t=\frac {a(a+1)}2-1$ and $\ell=b^2$, $a$,$b$ are non-negative integers. 
    Let $a'$ be such that $a=2a'$ if $a$ is even and $a=2a'+1$ if $a$ is odd.
    Let $\Sigma = b+a'$ and
    $$\delta = \begin{cases}
        b-a' & \mbox{ if } 2b> a \\
        a'-b & \mbox{ if } 2b\le a.
    \end{cases}$$
    The nilpotent orbit $\OO^\vee(c^\omega(J),\sigma)$ is parameterized by the partition
    $$\lambda = \begin{cases}
        (3,7,\dots,4\Sigma-1) \cup (1,5,\dots,4\delta-3) & \mbox{ if $a$ is even and } 2b> a \\
        (3,7,\dots,4\Sigma-1) \cup (3,7,\dots,4\delta-1) & \mbox{ if $a$ is even and } 2b\le a \\
        (1,5,\dots,4\Sigma+1) \cup (3,7,\dots,4\delta-5) & \mbox{ if $a$ is odd and } 2b> a \\
        (1,5,\dots,4\Sigma+1) \cup (1,5,\dots,4\delta+1) & \mbox{ if $a$ is odd and } 2b\le a.
    \end{cases}$$
    %
\end{enumerate}
\end{enumerate}

\subsection{Exceptional groups}\label{subsec:exceptionalcuspidal}

\paragraph{$G_2$} If $\bfG = G_2$, then $G^\vee = G_2(\CC)$, and $Z(G^\vee) = \{1\}$.
If $\omega = 1$ then $J$ is of the form $G_2$ and there are 4 choices for $\sigma$ as enumerated in Table \ref{table:cupsidalexceptional}.
In all cases 
$$\OO^\vee(c^\omega(J),\sigma) = G_2(a_1).$$

\paragraph{$F_4$} If $\bfG = F_4$, then $G^\vee = F_4(\CC)$, and $Z(G^\vee) = \{1\}$.
If $\omega = 1$ then $J$ is of the form $F_4$ and there are 7 choices for $\sigma$ as enumerated in Table \ref{table:cupsidalexceptional}.
In all cases 
$$\OO^\vee(c^\omega(J),\sigma) = F_4(a_3).$$

\paragraph{$E_6$} If $\bfG = E_6$, then $G^\vee = E_6(\CC)$, and $Z(G^\vee) = \{1,\zeta,\zeta^2\}$.
\begin{enumerate}
    \item If $\omega = 1$ then $J$ is of the form $E_6$ and there are 2 choices for $\sigma$ as enumerated in Table \ref{table:cupsidalexceptional}.
    In both cases 
$$\OO^\vee(c^\omega(J),\sigma) = D_4(a_1).$$
    \item If $\omega\in\{\zeta,\zeta^2\}$ then $J$ is of the form $~^3D_4$ and $\sigma = D_4[1]$ or $D_4[-1]$.
    In both cases 
$$\OO^\vee(c^\omega(J),\sigma) = E_6(a_3).$$
\end{enumerate}

\paragraph{$E_7$} If $\bfG = E_7$, then $G^\vee = E_7(\CC)$, and $Z(G^\vee) = \{1,-1\}$.
\begin{enumerate}
    \item If $\omega = 1$ then $J$ is of the form $E_7$ and there are 2 choices for $\sigma$ as enumerated in Table \ref{table:cupsidalexceptional}.
    In both cases 
$$\OO^\vee(c^\omega(J),\sigma) = A_4+A_1.$$
    \item If $\omega=-1$, then  $J$ is of the form $~^2\!E_6$. There are three cuspidal unipotent representations afforded by $J$: $^2\!E_6[1]$,  $^2\!E_6[\theta]$,  $^2\!E_6[\theta^2]$. 
    In all cases 
    $$\OO^\vee(c^\omega(J),\sigma) = E_7(a_5).$$
\end{enumerate}

\paragraph{$E_8$} If $\bfG = E_8$, then $G^\vee = E_8(\CC)$, and $Z(G^\vee) = \{1\}$.
If $\omega = 1$ then $J$ is of the form $E_8$ and there are 13 choices for $\sigma$ as enumerated in Table \ref{table:cupsidalexceptional}.
In all cases 
$$\OO^\vee(c^\omega(J),\sigma) = E_8(a_7).$$

\section{Proof of Proposition \ref{prop:main}}\label{sec:proof}

In this section, we will prove Proposition \ref{prop:main}. 

\subsection{Classical groups}
In each case we show that 
$$\overline{\mathbb L}(J,\WF(\sigma)) = d_A(\lambda,1)$$
where $\lambda$ is the partition parameterizing $\OO^\vee(c^\omega(J),\sigma)$.
We will use the machinery of \cite[Section 3.4]{Acharduality} to prove this equality.

\paragraph{$PGL(n)$}
Let $\omega\in \Omega\simeq C_n$ be of order $n$.
Let $J=\emptyset$.      
Then $\sigma = triv$, $\WF(\sigma) = \{0\}$, and $\OO^\vee(c^\omega(\emptyset),\sigma)$ is the principal orbit $\OO^\vee_{prin}$.
We need to show that 
$$\overline{\mathbb L}(\emptyset,\{0\}) = d_A(\OO^\vee_{prin},1).$$
But both sides are equal to $(\{0\},1)$ and so we have equality.

\paragraph{$SO(2n+1)$} 
Consider the cases $\omega = 1,-1$ simultaneously.
Fix integers $a,b$ as in section \ref{sec:typeB} to fix $J$ and hence $\sigma$.
By section \ref{sec:unipcusp}
$$\WF(\sigma) = (1,1,3,3,\dots,2a-1,2a-1)\times(1,1,3,3,\dots,2b-1,2b-1,2b+1).$$
Let $\delta,\Sigma,\lambda$ be as in section \ref{sec:typeB}.
We have that
\begin{align*}
    \lambda^t &= (\delta,\delta,\delta-1,\delta-1,\dots,1,1)\vee (\Sigma,\Sigma,\Sigma-1,\Sigma-1,\dots,1,1) \\
    &= \begin{cases}
        (2b,2b,2b-2,2b-2,\dots,2a,2a,2a-1,2a-1,\dots,1,1) & \mbox{ if } b\ge a \\
        (2a-1,2a-1,2a-3,2a-3,\dots,2b+1,2b+1,2b,2b,\dots,1,1) & \mbox{ if } a> b
    \end{cases}
\end{align*}
so $\pi(\lambda) = \emptyset$.
We also have
\begin{align*}
    d(\lambda) = (2b+1,2b-1,2b-1,\dots,1,1) \cup (2a-1,2a-1,\dots,1,1).
\end{align*}
Since $(1,1,3,3,\dots,2a-1,2a-1)$ only has parts with even multiplicity, 
$$\overline{\mathbb L}(J,\WF(\sigma)) = \hphantom{ }^{\langle (1,1,3,3,\dots,2a-1,2a-1)\rangle}d(\lambda) = \hphantom{ }^{\langle \emptyset \rangle}d(\lambda) = \hphantom{ }^{\langle\pi(\lambda)\rangle}d(\lambda) = d_A(\lambda,1)$$
where $\pi(\lambda)$ is the subpartition of $\lambda^t$ defined by Achar in (\cite{Acharduality}, Equation 8).

\paragraph{$PSp(2n)$} 
\begin{enumerate}
    \item Let $\omega = 1$.
Fix integers $a,b$ as in section \ref{sec:typeC} (1) to fix $J$ and hence $\sigma$.
By section \ref{sec:unipcusp}
$$\WF(\sigma) = (2,2,4,4,\dots,2a,2a)\times(2,2,4,4,\dots,2b,2b).$$
Let $\delta,\Sigma,\lambda$ be as in section \ref{sec:typeC} (1).
We have that
\begin{align*}
    \lambda^t 
    &= 
        (\delta,\delta-1,\delta-1,\dots,1,1)\vee (\Sigma+1,\Sigma,\Sigma,\dots,1,1) \\
    &= (2a+1,2a-1,2a-1,\dots,2b+1,2b+1,2b,2b,\dots,1,1)
\end{align*}
so $\pi(\lambda) = \emptyset$.
We also have
\begin{align*}
    d(\lambda) = (2a,2a,\dots,2b+2,2b+2,2b,2b,2b,2b,\dots,2,2,2,2).
\end{align*}
Since $(2,2,4,4,\dots,2a,2a)$ only has parts with even multiplicity, 
$$\overline{\mathbb L}(J,\WF(\sigma)) = \hphantom{ }^{\langle (2,2,4,4,\dots,2a,2a)\rangle}d(\lambda) = \hphantom{ }^{\langle \emptyset \rangle}d(\lambda) = \hphantom{ }^{\langle\pi(\lambda)\rangle}d(\lambda) = d_A(\lambda,1).$$

    \item Let $\omega = -1$.
    Fix integers $a,b$ as in section \ref{sec:typeC} (2) to fix $J$ and hence $\sigma$.
    Then 
    $$\WF(\sigma) = (2,2,4,4,\dots,2b,2b) \times (1,2,\dots,a) \times (2,2,4,4,\dots,2b,2b).$$
    Let $\delta,\Sigma,\lambda$ be as in section \ref{sec:typeC} (2).
    We have that
    \begin{align}
        \lambda^t &= \begin{cases}
            (\Sigma+1,\Sigma^4,\dots,1^4) \vee (\delta^3,(\delta-1)^4,\dots,1^4) & \mbox{ if $a$ is even  and } 2b\ge a \\
            (\Sigma+1,\Sigma^4,\dots,1^4) \vee (\delta,(\delta-1)^4,\dots,1^4) & \mbox{ if $a$ is even  and } 2b< a \\
            ((\Sigma+1)^3,\Sigma^4,\dots,1^4) \vee (\delta,(\delta-1)^4,\dots,1^4) & \mbox{ if $a$ is odd and } 2b\ge a \\
            ((\Sigma+1)^3,\Sigma^4,\dots,1^4) \vee (\delta^3,(\delta-1)^4,\dots,1^4) & \mbox{ if $a$ is odd and } 2b< a
        \end{cases} \\
        &=
        \begin{cases}
            (2b+1,(2b)^2,\dots,(a+1)^2,a^4,\dots,1^4) & \mbox{ if $a$ is even  and } 2b\ge a \\
            (a+1,(a-1)^4,\dots,(2b+1)^4,(2b)^4,\dots,1^4) & \mbox{ if $a$ is even  and } 2b< a \\
            (2b+1,(2b)^2,\dots,(a+1)^2,a^4,\dots,1^4) & \mbox{ if $a$ is odd and } 2b\ge a \\
            (a^3,(a-2)^4,\dots,(2b+1)^4,(2b)^4,\dots,1^4) & \mbox{ if $a$ is odd and } 2b< a.
        \end{cases}
    \end{align}
    Thus $\pi(\lambda) = \emptyset$ since all even parts of $\lambda^t$ have even multiplicity.
    Moreover 
    $$d(\lambda) = (2,2,4,4,\dots,2b,2b)\cup (1,1,2,2,\dots,a,a) \cup (2,2,4,4,\dots,2b,2b)$$
    in all cases.
    Thus
    \begin{align*}
        \overline{\mathbb L}(J,\WF(\sigma)) &= \overline{\mathbb L}(\tilde J,(2,2,\dots,2b,2b)\times (1,1,\dots,a,a)\cup(2,2,\dots,2b,2b)) \\
        &= \hphantom{ }^{\langle (2,2,4,4,\dots,2b,2b)\rangle}d(\lambda) = \hphantom{ }^{\langle \emptyset \rangle}d(\lambda) = \hphantom{ }^{\langle\pi(\lambda)\rangle}d(\lambda) = d_A(\lambda,1)
    \end{align*}
    where $\tilde J = C_l \times C_{t+1+l}$.
\end{enumerate}

\paragraph{$PSO(2n)$}
\begin{enumerate}
    \item Let $\omega \in \{1,\eta\}$.
    Fix integers $a,b$ as in section \ref{sec:typeD} (1), (2) to fix $J$ and hence $\sigma$.
By section \ref{sec:unipcusp}
$$\WF(\sigma) = (1,1,3,3,\dots,2a-1,2a-1)\times(1,1,3,3,\dots,2b-1,2b-1).$$
    Let $\delta,\Sigma,\lambda$ be as in section \ref{sec:typeD} (1).
We have that
\begin{align*}
    \lambda^t &= (\delta,\delta-1,\delta-1,\dots,1,1)\vee (\Sigma,\Sigma-1,\Sigma-1,\dots,1,1) \\
    &= (2a,2a-2,2a-2,\dots,2b,2b,2b-1,2b-1,\dots,1,1)
\end{align*}
so $\pi(\lambda) = \emptyset$ since all odd parts have even multiplicity.
We also have
\begin{align*}
    d(\lambda) = (2a-1,2a-1,\dots,2b+1,2b+1,2b-1,2b-1,2b-1,2b-1,\dots,1,1,1,1).
\end{align*}
Since $(1,1,3,3,\dots,2a-1,2a-1)$ only has parts with even multiplicity, 
$$\overline{\mathbb L}(J,\WF(\sigma)) = \hphantom{ }^{\langle (1,1,3,3,\dots,2a-1,2a-1)\rangle}d(\lambda) = \hphantom{ }^{\langle \emptyset \rangle}d(\lambda) = \hphantom{ }^{\pi(\lambda)}d(\lambda) = d_A(\lambda,1).$$

    \item Let $\omega \in \{\rho,\eta\rho\}$.
    We will treat the cases (i) and (ii) simultaneously.
    Fix integers $a,b$ as in section \ref{sec:typeD} (3) (ii) to fix $J$ and hence $\sigma$ (we treat (i) as the case with $b=0$).
    By section \ref{sec:unipcusp}
    $$\WF(\sigma) = (1,1,3,3,\dots,2b-1,2b-1) \times (1,2,\dots,a) \times (1,1,3,3,\dots,2b-1,2b-1).$$
    Let $\delta,\Sigma,\lambda$ be as in section \ref{sec:typeD} (3) (ii).
    We have that
    \begin{align}
        \lambda^t &= \begin{cases}
            (\Sigma^3,(\Sigma-1)^4,\dots,1^4) \vee (\delta,(\delta-1)^4,\dots,1^4) & \mbox{ if $a$ is even  and } 2b\le a \\
            (\Sigma^3,(\Sigma-1)^4,\dots,1^4) \vee (\delta^3,(\delta-1)^4,\dots,1^4) & \mbox{ if $a$ is even  and } 2b\le a \\
            (\Sigma+1,\Sigma^4,\dots,1^4) \vee (\delta,(\delta-1)^4,\dots,1^4) & \mbox{ if $a$ is odd and } 2b< a \\
            ((\Sigma+1)^3,\Sigma^4,\dots,1^4) \vee (\delta^3,(\delta-1)^4,\dots,1^4) & \mbox{ if $a$ is odd and } 2b< a
        \end{cases} \\
        &=
        \begin{cases}
            (2b+1,(2b)^2,\dots,(a+1)^2,a^4,\dots,1^4) & \mbox{ if $a$ is even  and } 2b\ge a \\
            (a+1,(a-1)^4,\dots,(2b+1)^4,(2b)^4,\dots,1^4) & \mbox{ if $a$ is even  and } 2b< a \\
            (2b+1,(2b)^2,\dots,(a+1)^2,a^4,\dots,1^4) & \mbox{ if $a$ is odd and } 2b\ge a \\
            (a^3,(a-2)^4,\dots,(2b+1)^4,(2b)^4,\dots,1^4) & \mbox{ if $a$ is odd and } 2b< a.
        \end{cases}
    \end{align}
    Thus $\pi(\lambda) = \emptyset$ since all even parts of $\lambda^t$ have even multiplicity.
    Moreover 
    $$d(\lambda) = (2,2,4,4,\dots,2b,2b)\cup (1,1,2,2,\dots,a,a) \cup (2,2,4,4,\dots,2b,2b)$$
    in all cases.
    Thus
    \begin{align*}
        \overline{\mathbb L}(J,\WF(\sigma)) &= \overline{\mathbb L}(\tilde J,(1,1,\dots,2b-1,2b-1)\times (1,1,\dots,a,a)\cup(1,1,\dots,2b-1,2b-1)) \\
        &= \hphantom{ }^{\langle (2,2,4,4,\dots,2b,2b)\rangle}d(\lambda) = \hphantom{ }^{\langle \emptyset \rangle}d(\lambda) = \hphantom{ }^{\langle\pi(\lambda)\rangle}d(\lambda) = d_A(\lambda,1)
    \end{align*}
    where $\tilde J = D_l \times D_{t+1+l}$.
\end{enumerate}

\subsection{Exceptional groups}

\paragraph{Split forms} Suppose that $\mathbf{G}$ is split, of exceptional type, and that $\omega = 1$. As can be seen in Section \ref{subsec:exceptionalcuspidal}, $J$ is always equal to $\Delta$. Thus,
$$\overline{\mathbb{L}}(J,\WF(\sigma))=(\WF(\sigma),1).$$
On the other hand, the nilpotent orbit $\OO^\vee := \OO^{\vee}(c^\omega(J),\sigma)$ is always special. Thus, 
$$d_A(\OO^{\vee},1) = (d(\OO^{\vee}),1)$$
by the general properties of $d_A$, see \cite[Section 3]{Acharduality}. So for Proposition \ref{prop:main} it suffices to show that
$$\WF(\sigma) = d(\OO^{\vee})$$
for all $\sigma$.
This follows by comparing the orbits in Table \ref{ta:exc-fin} and in section \ref{subsec:exceptionalcuspidal}.

\paragraph{Non-split forms of $E_6$}
Suppose $\bfG$ is of type $E_6$ and $\omega \in \{\zeta,\zeta^2\}$.
Then $J$ is of the form $\hphantom{ }^3D_4$, and $\WF(\sigma) = (1,1,3,3)$ for both $\sigma=D_4[1]$ and $\sigma = D_4[-1]$.
The orbit $(1,1,3,3)$ is the orbit $A_2$ in Bala-Carter notation.
Thus we need to show that 
$$\overline{\mathbb L}(J,(1,1,3,3)) = d_A(E_6(a_3),1).$$
We note that $E_6(a_3)$ is special and $d(E_6(a_3)) = A_2$ so we must show that
$$\overline{\mathbb L}(J,A_2) = (A_2,1).$$
Since $J\subseteq \Delta$ this follows from \cite[Proposition 2.30]{okada2021wavefront}.

\paragraph{Non-split forms of $E_7$}
Suppose $\mathbf{G}$ is of type $E_7$ and $\omega = -1$.
Then $J$ is of the form $\hphantom{ }^2E_6$, and $\WF(\sigma) = D_4(a_1)$ for all possible $\sigma$.
Thus we need to show that 
$$\overline{\mathbb L}(J,D_4(a_1)) = d_A(E_7(a_5),1).$$
We note that $E_7(a_5)$ is special and $d(E_7(a_5)) = D_4(a_1)$ so we must show that
$$\overline{\mathbb L}(J,D_4(a_1)) = (D_4(a_1),1).$$
Since $J\subseteq \Delta$ this follows from \cite[Proposition 2.30]{okada2021wavefront}.

\begin{sloppypar} \printbibliography[title={References}] \end{sloppypar}

@incollection{Kaletha2016,
title={The local Langlands conjecture for non-quasi-split groups},
author={Kaletha, T.},
booktitle={Families of automorphic forms and the trace formula},
series={Simons Symp.},
publisher={Springer},
year={2016},
pages={217--257}
}

@incollection{ABPS2017,
title={Conjectures about $p$-adic groups and their non commutaive geometry},
author={Aubert, A.-M. and Baum, P. and Plymen, R. and Solleveld, M.},
booktitle={Around Langlands correspondences},
series={Contemp. Math.},
publisher={Amer. Math. Soc., Providence, RI},
year={2017},
volume={691},
pages={5--21}
}

@article{Kottwitz1984,
 title={Stable trace formula: cuspidal tempered terms},
 author={Kottwitz, R.},
 journal={Duke Math. J.},
 year={1984},
 volume={51},
 issue={3},
 pages={611--650},
 }

@article{Opdam16,
 title={Spectral transfer morphisms for unipotent affine Hecke algebras},
 author={Opdam, E.},
 journal={Selecta Math. (N.S.)},
 year={2016},
 volume={22},
 issue={4},
 pages={2143--2207},
 }

@article{AMSol2018,
 author={Aubert, A.-M. and Moussaoui, A. and Solleveld, M.},
 title={Generalizations of the Springer correspondence and cuspidal Langlands parameters},
 journal={Manuscripta Math.},
 year={2018},
 volume={157},
 issue={1-2},
 pages={121--192},
 }

@article{Morris1996,
 title={Tamely ramified supercuspidal representations},
 author={Morris, L.},
 journal={Ann. Sci, \'Ecole Norm. Sup.},
 year={1996},
 volume={29},
 issue={5},
 pages={639-667},
 }

@incollection{HarishChandra1999,
title={Harish-Chandra Admissible invariant distributions on reductive p-adic groups. With a preface and notes by Stephen DeBacker and Paul J. Sally, Jr.},
author={Harish-Chandra},
series={University Lecture Series},
publisher={American Mathematical Society, Providence, RI},
year={1999},
volume={16},
pages={xiv+97 pp.}
}

@misc{FengOpdamSolleveld2021,

  author = {Feng, Y. and Opdam, E. and Solleveld, M.},
  
  title = {On formal degrees of unipotent representations},
  
  publisher = {arXiv:1910.07198},
  
  year = {2019},
}

@misc{AizGouSay,

  author = {Aizenbud, A. and Gourevitch, D. and Sayag, E.},
  
  title = {Irreducibility of wave front sets for depth zero cuspidal representations},
  
  publisher = {arXiv:2205.14695},
  
  year = {2022},
}

@misc{isog,

  author = {Okada, Emile Takahiro},
  
  title = {Wavefront sets and isogenies [In preparation]},
  
  publisher = {arXiv},
  
  year = {2022},
}

@misc{unipotent1,

  author = {Ciubotaru, D. and Mason-Brown, L. and Okada, E.},
  
  title = {Some Unipotent Arthur Packets for Reductive p-adic Groups I},
  
  publisher = {arXiv},
  
  year = {2021},
}

@article{FengOpdam2020,
 title={On a uniqueness property of supercuspidal unipotent representations},
 author={Feng, Y. and Opdam, E.},
 journal={Adv. Math.},
 year={2020},
 volume={375},
 pages={107406, 62 pp.},
 }

@article{Lusztig-IMRN,
 title={Classification of unipotent representations of simple $p$-adic groups},
 author={Lusztig, G.},
 journal={IMRN},
 year={1995},
 volume={11},
 pages={517--589},
 }

@article{Reeder-formal,
 title={Formal degrees and $L$-packets of unipotent discrete series representations of exceptional p-adic groups. With an appendix by Frank L\" ubeck},
 author={Reeder, M.},
 journal={J. Reine Angew. Math.},
 year={2000},
 volume={520},
 pages={37--93},
 }

@misc{Sol-LLC,
title={A local Langlands correspondence for unipotent representations},
author={Solleveld, M.},
archivPrefix={arXiv},
primaryClass={math.RT},
eprint={1806.11357},
year={2018}
}

@incollection{Vogan1993,
title={The local Langlands conjecture},
author={Vogan, D.A.},
booktitle={Representation theory of groups and algebras},
series={Contemp. Math.},
publisher={Amer. Math. Soc., Providence, RI},
year={1993},
volume={145},
pages={305--379}
}

@article{Wald18,
    title={Repr\' esentations de r\' eduction unipotente pour $SO(2n+1)$, III: exemples de fronts d'onde},
    author={Waldspurger, J.-L.},
    journal={Algebra Number Theory},
    year={2018},
    volume={12},
    pages={1107--1171},
   number={5}
}

@article {MW87,
    AUTHOR = {M\oe glin, C. and Waldspurger, J.-L.},
     TITLE = {Mod\`eles de {W}hittaker d\'{e}g\'{e}n\'{e}r\'{e}s pour des groupes
              {$p$}-adiques},
   JOURNAL = {Math. Z.},
  FJOURNAL = {Mathematische Zeitschrift},
    VOLUME = {196},
      YEAR = {1987},
    NUMBER = {3},
     PAGES = {427--452},
      ISSN = {0025-5874},
   MRCLASS = {22E50 (11F70 11S37)},
  MRNUMBER = {913667},
MRREVIEWER = {Mitsuo Sugiura},
       DOI = {10.1007/BF01200363},
       URL = {https://ezproxy-prd.bodleian.ox.ac.uk:2102/10.1007/BF01200363},
}

@article{Au,
    title={Dualit\' e dans le groupe de Grothendieck de la cat\' egorie des repr\' esentations lisses de longueur finie d'un groupe r\' eductif p-adique},
    author={Aubert, A.~M.},
    journal={Trans. Amer. Math. Soc.},
    year={1995},
    pages={2179--2189},
    volume={347},
    number={6},
}

@article{Pommerening,
    title={Über die unipotenten Klassen reduktiver Gruppen},
    author={K. Pommerening},
    journal={J. Algebra},
    year={1977},
    pages={373--398},
    %number={2},
}

@article{Pommerening2,
    title={Über die unipotenten Klassen reduktiver Gruppen II},
    author={K. Pommerening},
    journal={J. Algebra},
    year={1980},
    pages={525--536},
    volume={65},
    %number={2},
}

@article{Howe1974,
author = {Howe, R.},
journal = {Mathematische Annalen},
pages = {305-322},
title = {The Fourier Transform and Germs of Characters (Case of $Gl(n)$ over a $p$-Adic Field).},
volume = {208},
year = {1974},
}

@article {BarbaschVogan1985,
    AUTHOR = {Barbasch, D. and Vogan, D.},
     TITLE = {Unipotent representations of complex semisimple groups},
   JOURNAL = {Ann. of Math. (2)},
    VOLUME = {121},
    pages={41--110},
      YEAR = {1985},
    NUMBER = {1},
}

@book{Lusztig1984,
 author = {George Lusztig},
 publisher = {Princeton University Press},
 title = {Characters of Reductive Groups over a Finite Field. (AM-107)},
 year = {1984}
}

@book{Spaltenstein,
	year  = {1982},
	publisher = {Springer Berlin Heidelberg},
	author = {Spaltenstein, N.},
	title = {Classes Unipotentes et Sous-groupes de Borel}
}

@book{Carter1993,
   author={Carter, R.},
   title={Finite groups of Lie type},
   series={Wiley Classics Library},
   publisher={John Wiley \& Sons, Ltd.},
   date={1993},
}

@article {lusztigunip,
    AUTHOR = {Lusztig, George},
     TITLE = {A unipotent support for irreducible representations},
   JOURNAL = {Adv. Math.},
  FJOURNAL = {Advances in Mathematics},
    VOLUME = {94},
      YEAR = {1992},
    NUMBER = {2},
     PAGES = {139--179},
   %   ISSN = {0001-8708},
 %  MRCLASS = {20G05 (20G40)},
 % MRNUMBER = {1174392},
%MRREVIEWER = {Bhama Srinivasan},
  %     DOI = {10.1016/0001-8708(92)90035-J},
   %    URL = {https://ezproxy-prd.bodleian.ox.ac.uk:2102/10.1016/0001-8708(92)90035-J},
}

@article{Acharduality,
author={Achar, P.},
title={An order-reversing duality map for conjugacy classes in Lusztig's canonical quotient},
journal={Transform. Groups},
volume={8},
pages={107--145},
date={2003},
}

@misc{okada2021wavefront,
      title={The wavefront set over a maximal unramified field extension}, 
      author={Emile Takahiro Okada},
      year={2021},
      eprint={2107.10591},
      archivePrefix={arXiv},
      primaryClass={math.RT}
}

@article{Sommers2001,
  title={Lusztig's canonical quotient and generalized duality},
  author={E. Sommers},
  journal={J. Algebra},
  year={2001},
  volume={243},
  pages={790-812}
}

@incollection {kawanaka,
    AUTHOR = {Kawanaka, Noriaki},
     TITLE = {Shintani lifting and {G}el'fand-{G}raev representations},
 BOOKTITLE = {The {A}rcata {C}onference on {R}epresentations of {F}inite
              {G}roups ({A}rcata, {C}alif., 1986)},
    SERIES = {Proc. Sympos. Pure Math.},
    VOLUME = {47},
     PAGES = {147--163},
 PUBLISHER = {Amer. Math. Soc., Providence, RI},
      YEAR = {1987},
   MRCLASS = {22E50 (20G05)},
  MRNUMBER = {933357},
MRREVIEWER = {Marie-France Vign\'{e}ras},
}

@article {barmoy,
    AUTHOR = {Barbasch, Dan and Moy, Allen},
     TITLE = {Local character expansions},
   JOURNAL = {Ann. Sci. \'{E}cole Norm. Sup. (4)},
  FJOURNAL = {Annales Scientifiques de l'\'{E}cole Normale Sup\'{e}rieure. Quatri\`eme
              S\'{e}rie},
    VOLUME = {30},
      YEAR = {1997},
    NUMBER = {5},
     PAGES = {553--567},
    %  ISSN = {0012-9593},
   %MRCLASS = {22E50},
  %MRNUMBER = {1474804},
%MRREVIEWER = {Dihua Jiang},
 %      DOI = {10.1016/S0012-9593(97)89931-4},
  %     URL = {https://ezproxy-prd.bodleian.ox.ac.uk:2102/10.1016/S0012-9593(97)89931-4},
}

@article {rangarao,
    AUTHOR = {Ranga Rao, R.},
     TITLE = {Orbital integrals in reductive groups},
   JOURNAL = {Ann. of Math. (2)},
  FJOURNAL = {Annals of Mathematics. Second Series},
    VOLUME = {96},
      YEAR = {1972},
     PAGES = {505--510},
      ISSN = {0003-486X},
   MRCLASS = {22E45},
  MRNUMBER = {320232},
MRREVIEWER = {C. J. Henrich},
       DOI = {10.2307/1970822},
       URL = {https://ezproxy-prd.bodleian.ox.ac.uk:2102/10.2307/1970822},
}

@article {moyprasadJ,
    AUTHOR = {Moy, Allen and Prasad, Gopal},
     TITLE = {Jacquet functors and unrefined minimal {$K$}-types},
   JOURNAL = {Comment. Math. Helv.},
  FJOURNAL = {Commentarii Mathematici Helvetici},
    VOLUME = {71},
      YEAR = {1996},
    NUMBER = {1},
     PAGES = {98--121},
      ISSN = {0010-2571},
   MRCLASS = {22E50 (22E35)},
  MRNUMBER = {1371680},
MRREVIEWER = {Mark Reeder},
       DOI = {10.1007/BF02566411},
       URL = {https://ezproxy-prd.bodleian.ox.ac.uk:2102/10.1007/BF02566411},
}

@inproceedings {tits,
    AUTHOR = {Tits, J.},
     TITLE = {Reductive groups over local fields},
 BOOKTITLE = {Automorphic forms, representations and {$L$}-functions
              ({P}roc. {S}ympos. {P}ure {M}ath., {O}regon {S}tate {U}niv.,
              {C}orvallis, {O}re., 1977), {P}art 1},
    SERIES = {Proc. Sympos. Pure Math., XXXIII},
     PAGES = {29--69},
 PUBLISHER = {Amer. Math. Soc., Providence, R.I.},
      YEAR = {1979},
   MRCLASS = {20G25 (20G10)},
  MRNUMBER = {546588},
}

\end{document}